\documentclass[11pt]{article}

\usepackage{inputenc}
\usepackage{amsmath}
\usepackage{amsfonts}
\usepackage{amssymb}
\usepackage{amsthm}
\usepackage{graphicx}
\usepackage[T1]{fontenc}
\usepackage[english]{babel}
\usepackage{mathrsfs}
\usepackage{amsthm}
\usepackage[all]{xy}
\usepackage{geometry}
\usepackage{bm}
\usepackage{enumerate}
\usepackage{pgfplots}
\pgfplotsset{compat=newest}
\allowdisplaybreaks
\usepackage{xcolor}
\usepackage{appendix}
\usepackage{authblk}
\usepackage[colorlinks=true]{hyperref} 
\usepackage{cleveref}
\usepackage{dsfont}
\usepackage{mathtools}
\usepackage{algorithm,algorithmic}
\usepackage{pgf,tikz}
\usetikzlibrary{arrows}
\usepackage[font=small, labelfont=bf]{caption} 
\usepackage{graphicx, wrapfig, subcaption, setspace, booktabs}
\usepackage{stmaryrd}
\usepackage{cancel}
\usepackage{comment}

\newtheorem{theorem}{Theorem}

\newtheorem{definition}[theorem]{Definition}
\newtheorem{proposition}[theorem]{Proposition}

\newtheorem{definition-proposition}[theorem]{Definition-Proposition}

\newtheorem{remark}[theorem]{Remark}

\renewcommand{\d}{\partial}
\renewcommand{\P}{\mathcal{P}_\theta}
\renewcommand{\r}{\rho}

\renewcommand{\L}{\Lambda}
\renewcommand{\a}{\alpha}
\renewcommand{\b}{\beta}

\newcommand{\au}{\alpha_{\bm{u}}}
\newcommand{\bu}{\beta_{\bm{u}}}
\newcommand{\D}{\Delta}
\newcommand{\s}{\sigma}
\newcommand{\g}{\gamma}
\newcommand{\G}{\Gamma}
\newcommand{\vG}{{\mathbf{\Gamma}}}
\newcommand{\vone}{{\mathbf{1}}}
\newcommand{\vc}{{\mathbf{c}}}
\newcommand{\vu}{{\mathbf{u}}}

\newcommand{\vp}{{\mathbf{p}}}
\newcommand{\vq}{{\mathbf{q}}}
\newcommand{\vh}{{\mathbf{h}}}

\newcommand{\F}{{\mathcal{F}}}

\newcommand{\vr}{{\bm{\rho}}}
\newcommand{\vg}{{\bm{\g}}}
\newcommand{\U}{\mathcal{U}}
\newcommand{\R}{\mathbb{R}}
\newcommand{\Nm}{N_{\max}}
\newcommand{\e}{\varepsilon}
\newcommand{\Linf}{L^\infty(0,T,\mathcal{U}_{ad})}
\newcommand{\pt}[1]{ \overset{\textbf{.}}{#1} }
\newcommand{\N}{\mathbb{N}}
\newcommand{\Om}{\Omega}

\newcommand{\der}{\mathrm{d}}
\newcommand{\bv}{\operatorname{BV}(0,T)}

\newcommand{\J}{{\mathcal{J}_\theta}}
\newcommand{\Jin}{J_{\text{in}}}
\newcommand{\Jout}{J_{\text{out}}}

\DeclarePairedDelimiter\abs{\lvert}{\rvert}%
\DeclarePairedDelimiter\norm{\lVert}{\rVert}%
\newcommand{\tnorm}[1]{{\left\vert\kern-0.25ex\left\vert\kern-0.25ex\left\vert #1 
    \right\vert\kern-0.25ex\right\vert\kern-0.25ex\right\vert}}

\makeatletter
\let\oldabs\abs
\def\abs{\@ifstar{\oldabs}{\oldabs*}}
\let\oldnorm\norm
\def\norm{\@ifstar{\oldnorm}{\oldnorm*}}
\makeatother

\definecolor{zzttqq}{rgb}{0.6,0.2,0}
\definecolor{ffqqqq}{rgb}{1,0,0}
\definecolor{zzqqzz}{rgb}{0.6,0,0.6}
\definecolor{ffwwqq}{rgb}{1,0.4,0}
\definecolor{qqwwff}{rgb}{0,0.4,1}

\definecolor{qqzztt}{rgb}{0.0000,0.60000,0.20000}
\definecolor{zzttqq}{rgb}{0.60000,0.20000,0.0000}
\definecolor{zzqqzz}{rgb}{0.60000,0.0000,0.60000}
\definecolor{ffwwqq}{rgb}{1.0000,0.40000,0.0000}
\definecolor{qqwwff}{rgb}{0.0000,0.40000,1.0000}
\definecolor{ffqqqq}{rgb}{1.0000,0.0000,0.0000}
\definecolor{cqcqcq}{rgb}{0.75294,0.75294,0.75294}

\title{\bf Optimal scenario for road evacuation in an urban environment}
\author[1]{Mickael Bestard\footnote{\texttt{\small mickael.bestard@math.unistra.fr}} }
\author[2]{Emmanuel Franck \footnote{ \texttt{\small emmanuel.franck@inria.fr}} }
\author[1]{Laurent Navoret \footnote{ \texttt{\small laurent.navoret@math.unistra.fr}} }

\author[3,4]{Yannick Privat \footnote{\texttt{\small  yannick.privat@univ-lorraine.fr}} }

\affil[1]{IRMA, Universit\'e de Strasbourg, CNRS UMR 7501, Inria, 7 rue Ren\'e Descartes, 67084 Strasbourg, France. }
\affil[2]{Inria, IRMA, Universit\'e de Strasbourg, CNRS UMR 7501, 7 rue Ren\'e Descartes, 67084 Strasbourg, France. }
\affil[3]{Universit\'e de Lorraine, CNRS, Institut Elie Cartan de Lorraine, Inria, BP 70239 54506 Vand\oe uvre-l\`es-Nancy Cedex, France. }
\affil[4]{Institut Universitaire de France (IUF)}

\begin{document}
\maketitle

\begin{abstract}
    How to free a road from vehicle traffic as efficiently as possible and in a given time, in order to allow for example the passage of emergency vehicles? We are interested in this question which we reformulate as an optimal control problem. We consider a macroscopic road traffic model on networks, semi-discretized in space and decide to give ourselves the possibility to control the flow at junctions. Our target is to smooth the traffic along a given path within a fixed time. A parsimony constraint is imposed on the controls, in order to ensure that the optimal strategies are feasible in practice.

    We perform an analysis of the resulting optimal control problem, proving the existence of an optimal control and deriving optimality conditions, which we rewrite as a single functional equation. We then use this formulation to derive a new mixed algorithm interpreting it as a mix between two methods: a descent method combined with a fixed point method allowing global perturbations. We verify with numerical experiments the efficiency of this method on examples of graphs, first simple, then more complex. We highlight the efficiency of our approach by comparing it to standard methods. We propose an open source code implementing this approach in the Julia language.
\end{abstract}

\medskip



\section{Introduction}
\label{sec:intro}
\paragraph{Road traffic modeling.}
With the concentration of populations in cities where transport flows are constantly increasing, road traffic modeling has become a central issue in urban planning. Whether it is to configure traffic lights, to adapt public transport offers or to manage in an optimal way situations of high congestion. The knowledge of the behavior of a road network is an important issue for the management of a crisis in an urban environment. It is indeed necessary to quickly predict traffic, especially traffic jams, in order to organize rescue operations. Before creating decision support tools for crisis management involving road traffic, a first step is to develop a method for monitoring road traffic. 

In this work, we will focus on the evacuation of a traffic lane in finite time. It can be, for example, an axis that we wish to free in order to provide access to first aid. We approach this question in the form of an optimal control problem of a graph where each edge corresponds to a traffic lane. The main objective of this work is the determination of a prototype algorithm answering this question.



\paragraph{State of the art.} Depending on the situation of interest, several types of approaches co-exist to model this phenomenon, depending on the scale and the desired accuracy.
 At the microscopic scale, trajectories (speed and position) of each vehicle can be described through Lagrangian microscopic models such as car-following models~\cite{Gipps81, Wiedemann74, Treiber2000} or cellular automaton models~\cite{Biham92, Krug88, Nagel92}. Kinetic traffic models were also introduced in the 60's~\cite{Coscia07, delitala07, Fermo14, Materne02, Paveri75, Prigogine62,  Wegener96} and deal with vehicles on a mesoscopic scale in the form of distributions in position-velocity space. 
 
Finally, on a macroscopic scale, the vehicle flow is then considered as a continuous medium. This is the case of the famous model introduced by Lighthill, Whitham and Richards (denoted LWR in what follows)~\cite{Lighthill55, Richards56}. We will use it in the rest of this article, since the continuous description of the flow is particularly well adapted to the writing of a dynamic programming problem.
It is notable that this first order model introduced in the 50's was later extended to Euler-type systems with viscosity, of second order~\cite{payne71, whitham2011, Aw2000, Zhang02}. 

In this paper, we introduce and analyze an optimal control problem modeling an evacuation scenario for an axis belonging to a road network.  A close control model has already been introduced in~\cite{Gugat05}, but under restrictive assumptions about the flow regime in order to avoid congestion phenomena, which is what we want to avoid. 
It has been in particular highlighted that the flow may be not invertible 
with respect to control parameters, meaning that density at the nodes may not be reconstructed from boundary conditions.
In~\cite{goatin_speed_2016}, a method to manage variable speed limits combined with coordinated ramp metering within the framework of the Lighthill-Whitham-Richards (LWR) network model has been introduced. They consider a "first-discretize-then-optimize" numerical approach to solve it numerically.
Note that the optimal control problem we deal with in the following has some similarities with the one proposed by the authors of this article, but differs in several respects: our control variables aim to control not just the ramps metering, but an entire axis.
Furthermore, our approach to solving the underlying optimal control problem is quite different, in particular our use of optimality conditions leading to an efficient solution algorithm.


In~\cite{reilly_adjoint-based_nodate}, a similar control problem is discussed, which is addressed using a piecewise-linear flux framework driving to a much simplified adjoint system involving piecewise constant Jacobian matrices in time and space. This approach enables faster gradient computation, allowing them to tackle real-world simulations related to ramp-metering configuration. In our study, we opted to use the model introduced in~\cite{Coclite05} and dealt with the non-linear flux using automatic differentiation techniques to compute the gradient.
This kind of approach has also been used in~\cite{bayen_adjoint-based_2006} for air traffic flow with a modified LWR-based network model.

We also mention several contributions aiming at dealing with concrete or real time cases. The paper~\cite{fugenschuh_combinatorial_2006} deals with this kind of control problem for simplified nonlinear (essentially without considering congestions) and linear formulations of the model, considering large scale networks. 
In~\cite{kurzhanskiy_active_2010}, an algorithm for real-time traffic state estimation and short-term prediction is introduced.

Several other close control problems have been investigated in the past: in \cite{gottlich_modeling_2015}, the authors use switching controls to deal with traffic lights at an 8$\times$4 junction in a piecewise linear flow framework.
The paper~\cite{ancona_optimization_2018} deals with an optimal control problem of the same kind, formulated from the continuous LWR model, solved numerically using an heuristic random parameters search.
We refer to \cite{ancona_analysis_2017,bayen_control_2022} for a broad view on this topic.

Regarding now controlled microscopic models, let us mention \cite{Malena22} where an approach based on model predictive control (MPD) is developed and \cite{Baumgart22} using a reinforcement learning algorithm. In the book \cite{Treiber13}, the last chapter is devoted to the modeling of various optimization problems related to road traffic. There has also been some research into the optimal switching of the traffic lights to maximize the traffic flow using a mixed-integer model~\cite{GOTTLICH201536}. Other works have looked at traffic regulation by comparing and discussing the use of lights and circles~\cite{chitour_traffic_2005}.

Concerning now the control of mesoscopic and macroscopic models, there are software contributions seeking in particular to bring real time answers and guidance tools in case of accident. Let us also mention \cite{ Jayakrishnan94,Leonard89} dealing in part with traffic prediction and control problems. 

This article uses a controlled model close to the one studied in \cite{Gugat05} which is devoted to the control of the macroscopic LWR model on a directed graph using the framework introduced in \cite{Coclite05}. However, we are interested in taking into account some properties of the model that were not addressed in \cite{Gugat05}, such as congestion phenomena.

Hence, we will not focus on individual paths since we are interested in the overall dynamics with a focus on congestion. Particular attention is paid to the modeling of junctions as they introduce a nonlinear coupling between roads.

\paragraph{Main objectives.}
Our goal is to design a numerical method to tackle the control the flow of vehicles in all fluid regimes, including saturated regimes/congestions, and using dynamical barriers at each road in the graph. 

We thus propose a road traffic model in the form of a controlled graph at each junction. This models an urban area structured by roads. For practical reasons, even if we wish to take into account possible congestion/shocks in the system, we position ourselves in a differentiable framework by adding weak diffusion terms in the system, through a semi-discretization in space. Hence, the resulting problem consists in minimizing the vehicle density at the final time on a given route under ODE constraint, obtained by finite volume semi-discretization of the LWR system. We then enrich this problem with a constraint on the number of simultaneously active roadblocks in order to take into account staffing issues. We use this analysis to build a new algorithm adapted to this problem, mixing a classical gradient approach with a fixed-point method allowing non-local perturbations. One of the advantages of this approach is that the constructed sequence can escape from a minimum well and survey a larger part of the admissible set than a standard gradient type algorithm.

\paragraph{Plan of this article.} 
While the macroscopic model describing the flow is very standard (Lighthill-Whitham-Richards, 1955) the problem of flow distribution at junctions, inspired by more recent works~\cite{GRB07,TFN}, realizes a nonlinear coupling between the different edges. 
The distribution of vehicles at junctions is thus modeled by an optimal allocation process that depends on the maximum possible flows at the junctions, through a linear programming (LP) problem targeting the maximization of incoming flows. In order to influence the traffic flow, we introduce control functions that are defined at each road entrance and acts as barriers by weighting the capacity of a road leaving a junction to admit new vehicles. 

Section~\ref{sec:roads} is devoted to the description of the traffic flow model we consider, and to the optimal control problem modeling the best evacuation scenario for a road. 

After investigating the existence of optimal control in Section~\ref{sec:wellposedness}, we then derive necessary optimality conditions that we reformulate in an exploitable way in Section~\ref{sec:optimCond}. We hence introduce an optimization algorithm in Section~\ref{sec:num_method}, based on a hybrid combination of a variable step gradient method and a well-chosen fixed point method translating the optimality conditions and allowing global perturbations. Finally, in Section~\ref{sec:num_results}, we illustrate the efficiency of the introduced method using examples of various graphs modeling, in particular, a traffic circle or a main road surrounded by satellite roads. 

Our contribution can be summarized in the form of a computational open source code allowing to process complex graphs, but nevertheless of rather small dimension. This code, written in the Julia language,  is available and usable at the following link:

\begin{center}
    \url{https://github.com/mickaelbestard/TRoN.jl}
\end{center}

\paragraph{Notations.} Throughout this article, we will use the following notations:
\begin{itemize}
\item For $x\in \R^n$, $x_+$ will denote the positive part of $x$, namely $x_+=\max\{x,0_{\R^n}\}$, the max being understood component by component;
\item $\bv$ denotes the space of all functions of bounded variation on $(0,T)$;
\item $W^{1,\infty}(0,T)$ denotes the space of all functions $f$ in $L^\infty (0,T)$ whose gradient in the sense of distributions also belongs to $L^\infty(0,T)$;
\item $\Vert \cdot \Vert_{\R^d}$ (resp. $\langle \cdot,\cdot\rangle_{\R^d}$) stands for the standard Euclidean norm (resp. inner product) in $\R^d$;
\item Let $F:\R^p\times \R^q\to \R^m$ and $(x_0,y_0)\in \R^p\times \R^q$. We will denote by $\d_x F(x_0,y_0)$, the Jacobian matrix of the mapping $x\mapsto F(x,y_0)$. When no ambiguity is possible, we will simply denote it by $\d_x F$;
\item $N_r$: number of roads in the model;
\item $N_c$ number of mesh cells per road for the first-order Finite Volumes scheme;
\item $N_T$: number of time discretization steps in the numerical schemes;
\end{itemize}

\section{A controlled model of trafic flow}


In this section, we first present the traffic flow model on a road network and its semi-discretization. Then we introduce the precise control model:  the control at junctions will translates regulation action using traffic signs or traffic lights.

\subsection{Traffic dynamics on network without control}\label{sec:roads}

The road network is a directed graph  of $N_r$ roads, with $N_r\in \N^*$, where each edge corresponds to a road and each vertex to a junction. The roads will all be denoted as real intervals $[a_i,b_i]$ for some $i\in \llbracket 1,N_r\rrbracket$. Of course, this writing does not determine the topology of the graph. It is necessary to define, for each junction, the indices of the incoming and outgoing roads in order to characterize the directed graph completely. 
 

 We are interested in the time evolution of the densities on each road. On the $i$-th road, the evolution of the density $\rho_i(t,x)$ 
is provided by the so-called LWR model
\begin{equation}
     \tag{LWR}
   \begin{cases}
        \d_t \rho_i(t,x) + \d_x f_i(\rho_i(t,x)) = 0 , & (t,x) \in (0,T)\times [a_i, b_i], \\ 
        \rho_i(0,x) = \rho_i^0(x), & x\in[a_i, b_i], \\
        f_i(\rho_i(t,a_i)) = \g^L_i(t), & t\in (0,T), \\
        f_i(\rho_i(t,b_i)) = \g^R_i(t), & t\in (0,T),
    \end{cases}
    \label{lwr-c}
\end{equation}
where the flux $f_i(\r)$ is given by
\begin{equation}\label{eq:fluxlwr}
    f_i(\r) = \r\,v_i^{\max}\left(1-\frac{\r}{\r_i^{\max}}\right), 
\end{equation}
where $v_i^{\max}$ and $\r^{\max}_i$ denotes respectively the maximal velocity and density allowed on the $i$-th road. We will consider without loss of generality that $\r_i^{\max} = v_i^{\max} = 1$. The initial density is denoted $\rho_i^0(x)$ and $\vg_i^L$ and $\vg_i^R$ are  functions allowing to prescribe the flux at the left and right boundaries of the interval. 


The fluxes $\vg$ at junctions are determined using the model introduced in~\cite{Coclite05}. Considering a junction $J$,  the sets of indices corresponding to the in-going and out-going roads are denoted $\Jin$ and $\Jout$. We assume that there is a statistical behavior matrix $A_J$ defined by
\begin{equation*}
	A_J := \left(\a_{ji}\right)_{(j,i) \in \Jout \times \Jin},\quad 0 < \a_{ji} < 1,\quad \sum_{j \in \Jout}\a_{ji}=1,
\end{equation*}
with $\a_{ji}$ the proportion of vehicles going to the $j$-th outgoing road among those coming from the $i$-th ingoing road. Outgoing fluxes have to satisfy the relation 
\begin{equation}(\g_j^L)_{j \in \Jout} =  A_J (\g_i^R)_{i \in \Jin},
\label{eq:relation_fluxes}
\end{equation} and we have  the balance between in-going and out-going fluxes: 
\begin{equation*}
	\sum_{i \in \Jin} \g_i^R = \sum_{j \in \Jout} \g_j^L.
\end{equation*}
Let us note that we have natural constraints on the fluxes. Indeed the LWR flux function $f_i$ reaches its maximum at $\rho_i^{\max}/2$. Therefore, ingoing fluxes $\g_i^R$ have to be smaller than the following upper bound:
\begin{equation*}\label{gin}
	\g_i^{R,\max} = f_i(\min\{\r_i(t,b_i),\rho_i^{\max}/2\}),
\end{equation*}
which takes into account reduced capacities when the density is lower than $\rho_i^{\max}/2$ in the ingoing road. Similarly outgoing fluxes $\g_j^L$ have to be smaller than the  upper bound:
\begin{equation*}\label{gout}
	\g_j^{L,\max} = f_j(\max\{\r_j(t,a_j),\rho_j^{\max}/2\}),
\end{equation*}
which takes into account reduced capacities when the density is larger than $\rho_i^{\max}/2$ in the outgoing road. We refer to \cite{Coclite05} for details. Consequently, the fluxes belong to following set:
\begin{equation*}
\Om_J = \left\{ \left((\g_i^R)_{j \in \Jin},(\g_j^L)_{i \in \Jout}\right) \in \prod_{i \in \Jin} \left[0,\g_i^{R,\max}\right] \prod_{j \in \Jout} \left[0,\g_j^{L,\max}\right] \quad \text{with} \quad (\g_j^L)_{j \in \Jout} =  A_J (\g_i^R)_{i \in \Jin} \right\}
\end{equation*}
Then we assume that drivers succeed in maximizing the total flow. Consequently the fluxes are solution to the following Linear Programming (LP) problem: 
\begin{equation}\label{pb:mainLP}
    \max_{ \left((\g_i^R)_{j \in \Jin},(\g_j^L)_{j \in \Jout}\right)\in\Om_J}  \ \sum_{i \in \Jin} \g_i^R,
\end{equation}
This definition is valid at any junction of the network. Note that this problem does not have necessarily a unique solution in the case where there are more incoming roads than outgoing roads. In that case, a priority modeling assumption should be added to select one solution. We refer to \cite{bretti2006numerical} for more details. In the sequel, we suppose that this linear programming problem have a unique solution solution and we formally write:
\begin{equation*}
\vg(t) = \phi^{LP}(\vr(t)), 
\end{equation*}
for the simultaneous resolution of the linear programming problems at all the junctions of the network. Note that the dependency on $\vr(t)$ results from the definition of the upper bounds involved in the definition of the sets $\Om_J$. 

In the following, we will only consider networks with at most two ingoing roads and two outgoing roads, i.e. either $1\times 1$, $1\times 2$, $2\times 1$ or $2\times 2$ junctions. This gives existence and stability of the LWR Cauchy problem \cite{Coclite05}, and has the advantage that the function $\phi^{LP}$ can be explicitly provided \cite{Shi16} while it already enables to model a huge variety of road networks.

\subsection{Control at junctions}\label{sec:introCont}

We introduce the vector of controls $\mathbf{u}= \left(u_i(\cdot)\right)_{1\leq i\leq N_r}\in L^\infty([0,T],\R^{N_r})$ at every road entrance. Note that we implicitly assume that every road junction is controlled, but our model allows without any difficulty to neutralize some controls in order to model the fact that only some junctions are controlled.

We interpret each control $u_j$ as a rate, assuming that at each time $t\in(0,T)$, the maximum flow out a junction and going into road $j$ is weighted by a coefficient $u_j(t)\in [0,1]$.
This allow in particular to keep valid all well-posedness considerations mentioned in \cite{Coclite05}. At junction $J$, we thus define the following polytope of constraints 
\begin{align*}
\Om_J(\vu)  &= \Big\{ \left((\g_i^R)_{j \in \Jin},(\g_j^L)_{j \in \Jout}\right) \in \prod_{i \in \Jin} \left[0,\g_i^{R,\max}\right]   \prod_{j \in \Jout} \left[0,(1-u_j)\g_j^{L,\max}\right]  \\
&\hspace{8cm}\quad \text{with}\quad \ (\g_j^L)_{j \in \Jout} =  A_J (\g_i^R)_{i \in \Jin} \Big\}
\end{align*}
This initial model is not completely satisfactory since a full control on one outgoing road would result on zero outgoing fluxes for all the ingoing roads. Indeed relation \eqref{eq:relation_fluxes} implies that ingoing fluxes are linear combination of outgoing fluxes with positive weights. This is definitely not the desired behavior as we would expect that the trafic flow would be deviated to the uncontrolled outgoing roads. To solve this problem, we choose to make the statistical behavior matrix $A_J$ also dependent on the control $\mathbf{u}$. More precisely, we ask that 
\begin{equation*}
\alpha_{ji}(\mathbf{u}) = 0, \text{ if }u_j = 0.
\end{equation*} 
In Appendix~\ref{append:modelCont}, we provide explicit constructions of such functions for $2\times 1$ and $2\times 2$ junctions. $1\times 1$ and $1\times 2$ junctions do not require such modification.

Like in the previous section, the simultaneous resolution of the Linear Porgramming problems at each junctions junction is now denoted:
\begin{equation*}
    \vg = \phi^{LP}(\vr,\vu)
\end{equation*}
The explicit expressions of $\phi^{LP}$ in the set of cases we deal with are provided in Appendix~\ref{append:modelCont}. Beyond the explicit expression of $\phi^{LP}(\vr,\vu)$, what matters is that $\phi^{LP}$ is a Lipschitz function with respect to $(\vr,\vu)$. In the following, in order to use tools of differentiable optimization, we will consider a $C^1$ approximation of this function. This issue is commented at the end of Appendix~\ref{append:modelCont}. According to these comments, we will assume from now on:
\begin{equation}\label{hyp:phiLP}\tag{H$_{\phi^{LP}}$}
\text{The function $\phi^{LP}$ is Lipschitz, and $C^1$ with respect to its second variable $\vu$. }
\end{equation}

\subsection{Semi-discretized model} For algorithmic efficiency reasons, we prefer to be able to define the sensitivity of the different data of the problem with respect to the control. Since the model \eqref{lwr-c} is known to generate possible irregularities in the form of shocks, we have decided to introduce regularity through a semi-discretization of the model in space.
The relevance of this choice will be discussed in the concluding section.

Hence, we discretize the model~\eqref{lwr-c} with a first-order Finite Volume (FV) scheme. We consider $N_c\in \N^*$ mesh cells per road: the discretization points on the $i$-th roads are denoted $a_i = x_{i,1/2} < x_{i,3/2} < \ldots < x_{i,N_c-1/2} = b_i$ and the space steps $\D x_{i,j} = x_{i,j+1/2} - x_{i,j-1/2}$.  
Then the discrete densities and the boundary fluxes are denoted:
\begin{align*}
    \vr(t) &= \big(\r_{i,j}(t)\big)_{\substack{1\leq i\leq N_r\\ 1\leq j\leq N_c}} = \left( \frac{1}{\D x_{i,j}}\int_{x_{i,j-\frac{1}{2}}}^{x_{i,j+\frac{1}{2}}}\r(t,x)\,\der x \right)_{\substack{1\leq i\leq N_r\\ 1\leq j\leq N_c}},\\
    \vg(t) &= \left( \gamma_1^L(t), \gamma_1^R(t),\ldots, \gamma_{N_r}^L(t), \gamma_{N_r}^R(t)\right).
\end{align*}
With a slight abuse of notation, we have written similarly the discrete variable as the continuous one, since we will essentially work on the semi-discretized model. Then the semi-discretized dynamics is given by the differential equation
\begin{equation}
    \label{eq:LWR-sd}\tag{LWR-sd}
    \begin{cases}
        \dfrac{\der\vr}{\der t}\big(t\big) = f^{FV}\left(\vr(t),\vg(t)\right), & t\in(0,T) \\
		\vg(t) = \phi^{LP}\left(\vr(t), \mathbf{u}(t))\right), & t\in(0,T) \\
		\vr(0) = \vr_0, & 
    \end{cases}
\end{equation}
where 
the finite volume flow at the $j$-th mesh cell of the $i$-th road reads
\begin{equation}\label{def:fFV}
    f^{FV}(\vr,\vg)_{ij} = \begin{cases}
        \frac{-1}{\D x_{ij}}\left(\F_i(\r_{i,j},\r_{i,j+1}) -  \g_i^L\right) ,& \text{if } j=1, \\ 
        \frac{-1}{\D x_{ij}}\left(\F_i(\r_{i,j},\r_{i,j+1}) - \F_i(\r_{i,j-1},\r_{i,j})\right) ,& \text{if } 2\leq j\leq N_c-1, \\ 
        \frac{-1}{\D x_{ij}}\left(\g_i^R - \F_i(\r_{i,j-1},\r_{i,j})\right) ,& \text{if } j=N_c, \\ 
    \end{cases}
\end{equation}
with $\F_i(u,v)$ the so-called local Lax-Friedrich numerical flux given by
\begin{equation*}
    \F_i(u,v) = \frac{f_i(u) + f_i(v)}{2} - \max{\left\{|f_i'(u)|, |f_i'(v)|\right\}}\frac{(v - u)}{2}.
\end{equation*}
We refer to \cite{Leveque92} for an introduction to Finite Volume approximations.

\subsection{Conclusion: an optimal control problem}


We are interested in an approximate controllability problem which consists in emptying a given route \emph{as much as possible} for a fixed end time $T>0$.  Let us introduce $\chi_{\text{path}}\subset \llbracket 1,N_r\rrbracket$, a set of indices corresponding to the route we wish to empty in a time $T$.

We would like to minimize a functional with respect to the control $\vu$, representing the sum of all the densities on this path, in other words
\begin{align*}
  C_T(\vu) &= \sum_{i\in\chi_{\text{path}}} \sum_{j=1}^{N_c}\r_{i,j}(T;\vu) = \vc\cdot \vr(T;\vu),
\end{align*}
where $\vr(t;\vu)$ denotes the solution to the controlled system~\eqref{eq:LWR-sd} and $\vc=(c_i)_{1\leq i\leq N_r}$ the vector defined by
$$
c_i=1 \text{ if }i\in \chi_{\text{path}}\quad \text{and}\quad c_i=0\text{ else}.
$$ 


Of course, it is necessary to introduce a certain number of constraints on the sought controls, in accordance with the model under consideration, and to model that the obtained control is feasible in practice. We are thus driven to consider the following constraints:
\begin{enumerate}[(i)]
\item $0\leq u_i(\cdot) \leq 1$, meaning that at each time, the control is a vehicle acceptance rate on a road;
\item $\sum_{i=1}^{N_r} u_i(\cdot) \leq N_{\max}$: we impose at each time a maximum number of active controls in order to take into account the staff required for roadblocking. 
\item Regularity: we will assume that $\vu$ is of bounded variation, and write $\vu \in \operatorname{BV}([0,T];\R^{N_r})$.
This constraint involves the {\it total variation}\footnote{Given a function $f$ belonging to $L^1([0,T])$, the total variation of $f$ in $[0,T]$ is defined as
$$
\operatorname{TV}(f)= \sup \left\{ \int_0^T  f(t) \phi' ( t ) \, dt , \ \phi \in C_c^1 ([0,T]), \  \Vert \phi \Vert_{L^\infty ( [0,T] )} \leq 1 \right\}.
$$
} of the control and models that the roadblock is supposed not to "blink" over time. From a control point of view, we aim at avoiding the so-called {\it chattering phenomenon}.
\end{enumerate}
The first constraint above will be included in the set of admissible controls. Concerning the other two constraints, we have chosen to include them as penalty/regularization terms in the functional. Of course, other choices would be quite possible and relevant.

Let $N_{\max}\in \N^*$ be an integer standing for the maximal number of active controls for this problem and $\theta=(\theta_S, \theta_B)\in \R^2$ be two non-negative parameters.
According to all the considerations above, the optimal control problem we will investigate reads:
\begin{equation}\label{eq:control} \tag{$\P$}
\boxed{\inf_{\vu \in \U_{ad}}{\J(\vu)},}
\end{equation}
where the admissible set of controls is defined by 
\begin{equation}\label{def:Uad}
\U_{ad}=L^\infty([0,T],[0,1]^{N_r}),
\end{equation}
and the regularized cost functional $\J$ writes
\begin{equation}
    \J(\vu) = C_T(\vu) + \frac{\theta_S}{2} S(\vu) + \theta_B B(\vu),
\end{equation}
where $S(\vu)$ denotes the regularizing term modeling the limitation on the number of active controls and $B(\vu)$ is the total variation of $\vu$ in time, namely 
\begin{equation}\label{eq:staffing}
    S(\vu) = \int_0^T\left(\sum_{i=1}^{N_r} u_i(t)-\Nm\right)_+^2\,dt \quad \text{and}\quad  B(\vu) = \sum_{i=1}^{N_r}\operatorname{TV}(u_i).
\end{equation}

\section{Analysis of the optimal control problem \eqref{eq:control}}
\label{sec:control}

In this section, we will investigate the well-posedness and derive the optimality conditions of Problem~\eqref{eq:control}. These conditions will form the basis of the algorithms used in the rest of this study. 

\subsection{Well-posedness of Problem~\eqref{eq:control}}\label{sec:wellposedness}

We follow the so-called direct method of calculus of variations. The key point of the following result is the establishment of uniform estimates of $\vr$ with respect to the control variable $\vu$, in the $W^{1,\infty}(0,T)$ norm. 

\begin{theorem}\label{theo:existOCP}
Let us assume that $\theta_B$ is positive. Then, Problem~\eqref{eq:control} has a solution.
\end{theorem}

\begin{proof}
    Let $(\vu_n)_{n\in \N}$ be a minimizing sequence for Problem~\eqref{eq:control}. Observe first that, by minimality, the sequence $(\J(\vu_n))_{n\in \N}$ is bounded, and therefore, so is the total variation $(B(\vu_n))_{n\in \N}$. 
We infer that, up to a subsequence, $(\vu_n)_{n\in \N}$ converges in $L^1(0,T)$ and almost everywhere toward an element $\vu_*\in \bv$. Since the class $\U_{ad}$ is closed for this convergence, we moreover get that $\vu_*\in L^\infty([0,T],[0,1]^{N_r})$.

For the sake of readability, we will denote similarly a sequence and any subsequence with a slight abuse of notation.

In what follows, it is convenient to introduce the function $g:\R^{N_c\times N_r}\times [0,1]^{N_r}$ defined by $g(\vr,\vu) = f^{FV}(\vr,\phi^{LP}(\vr,\vu))$, so that $\vr$ solves the ODE system $\vr' = g(\vr,\vu)$. Let us set $\vr_n := \vr(\cdot\,;\vu_n)$. 

\paragraph{Step 1: the function $g$ is continuous.}
Indeed, note that the function $\phi^{LP}$ is continuous (and even Lipschitz), according to \eqref{hyp:phiLP}. We also refer to the comments of modeling issues at the end of Section~\ref{sec:introCont}.  According to \eqref{def:fFV}, $f^{FV}$ has therefore the same regularity as the numerical flow $\vr\mapsto \F(\vr)$, defined by 
$$
\F_i^{j,j+1}(\vr) = \frac{f_i^j+f_i^{j+1}}{2} - \frac{\max\{|df_i^j|,|df_i^{j+1}|\}}{2}\left(\rho_i^{j+1}-\rho_i^j\right),
$$
with $f_i^j = \rho_i^j(1-\rho_i^j)$ and $df_i^j = 1-2\rho_i^j$, which is obviously continuous.

\paragraph{Step 2: the sequence $(\vr_n)_{n\in \N}$ is uniformly bounded in $W^{1,\infty}(0,T)$.} 
 For $\vu\in \U_{ad}$ given, let us introduce the function $\mathcal{L} : t\mapsto \frac{1}{2}\Vert \rho(t)\Vert_{\R^{N_c\times N_r}}^2$. Up to standard renormalizations, we will suppose in this proof that $v_i^{\max} = \r_i^{\max}=1$ and that $\Delta x_i$ does not depend on the index $i$. According to the chain rule, $\mathcal{L}$ is differentiable, and its derivative reads\footnote{we drop the time dependancy notation for readability.}
    \begin{eqnarray*}
        \mathcal{L}'(t) &=& \langle \vr'(t),\vr(t) \rangle_{\R^{N_c\times N_r}} = \langle g(\vr(t),\vu(t)), \vr(t) \rangle_{\R^{N_c\times N_r}} \\
         &= &-\frac{1}{\Delta x}\sum_{i=1}^{N_r}\sum_{j=1}^{N_c}\left(\F_i^{j,j+1}(\vr)- \F_i^{j-1,j}(\vr)\right)\r_i^j\\ 
         &=& -\frac{1}{\Delta x}\sum_{i=1}^{N_r}\left(\g_i^R\r_i^{N_c}+\sum_{j=1}^{N_c-1}\F_i^{j,j+1}(\vr)\r_i^j - \g_i^L\r_i^1 - \sum_{j=2}^{N_c}\F_i^{j-1,j}(\vr)\r_i^j \right)\\ 
         &=& \frac{1}{\Delta x}\sum_{i=1}^{N_r}\left(\g_i^L\r_i^1-\g_i^R\r_i^{N_c} + \sum_{j=1}^{N_c-1}\F_i^{j, j+1}(\vr)\left(\r_i^{j+1}-\r_i^{j}\right)\right). 
    \end{eqnarray*}
For a given $i\in \llbracket 1,N_r\rrbracket$, and $j\in\llbracket 1,N_c-1\rrbracket$, we have
    \begin{align*}
        \F_i^{j,j+1}(\vr)\left(\r_i^{j+1}-\r_i^j\right) &= \left( \frac{f_i^j+f_i^{j+1}}{2} - \frac{\max\{|df_i^j|,|df_i^{j+1}|\}}{2}\left(\rho_i^{j+1}-\rho_i^j\right) \right)\left(\r_i^{j+1}-\r_i^j\right)\\
        &\leq \left(f_i(\s_i) + (\r_i^{j+1}-\r_i^j) \right) \left(\r_i^{j+1}-\r_i^j\right),
    \end{align*}
    since $f_i^j\leq f_i(\s_i)= 1/4$  according to the explicit expression~\eqref{eq:fluxlwr} of $f_i$, and $|df_i^j|\leq v_i^{\max}=1$. Since the local Lax-Friedrich numerical scheme is monotone, we infer that $\rho_i^j$ is lower than $\rho_{\max} = 1$. We then have
    \begin{align*}
        \F_i^{j,j+1}(\vr)\left(\r_i^{j+1}-\r_i^j\right) &\leq 2 f_i(\s_i) + \left(\r_i^{j+1}-\r_i^j\right)^2 = \frac{1}{2} + \left(\r_i^{j+1}-\r_i^j\right)^2 \\
&\leq\frac{1}{2} + (\r_i^j)^2 + (\r_i^{j+1})^2 - 2\r_i^j\r_i^{j+1} \\ 
       & \leq \frac{1}{2} + (\r_i^j)^2 + (\r_i^{j+1})^2, 
    \end{align*}
    by positivity of the  $\r_i^j$. Using the majoration $ \g_i^L\r_i^1-\g_i^R\r_i^{N_c} \leq f_i(\s_i)\times 2 = 1/2$, all the calculations above yield 
    \begin{align*}
        \mathcal{L}'(t) & \leq   \frac{1}{\Delta x}\sum_{i=1}^{N_r} \left(\frac12 + \sum_{j=1}^{N_c-1} \left(\frac12+ (\r_i^j)^2 + (\r_i^{j+1})^2 \right) \right) \\ 
                        &\leq  \frac{1}{\Delta x} \left( \frac{N_r N_c}{2} + || \vr(t) || ^2 \right).
    \end{align*}
   We infer the existence of two positive numbers $\bar \alpha$, $\bar\beta$ that do not depend on $\vu$, such that $\mathcal{L}'(t) \leq \bar\alpha \mathcal{L}(t) + \bar\beta$ for a.e. $t\in(0,T)$.
 By using a Gr\"onwall-type inequality, we get
    \begin{equation}\label{eq:2325}
        \mathcal{L}(t) \leq \mathcal{L}(0)e^{\bar \alpha t} + \int_0^t \beta e^{\bar \alpha(t-s)}\, ds = \left(\mathcal{L}(0) + \frac{\bar \beta}{\bar \alpha}(1-e^{-\bar \alpha t})\right) e^{\bar \alpha t},
    \end{equation}

\paragraph{Step 3: conclusion.}
   According to \eqref{eq:2325}, the sequence $(\vr_n)_{n\in \N}$ is bounded in $L^\infty(0,T)$. Since 
$\vr_n' = g(\vr_n,\vu_n)$ a.e. in $(0,T)$ and $g$ is continuous, it follows that $(\vr_n)_{n\in \N}$ is uniformly bounded in $W^{1,\infty}(0,T)$.
 Therefore, by using the Arzela-Ascoli theorem, this sequence converges up to a subsequence in $C^0([0,T])$ toward an element $\vr_* \in W^{1,\infty}(0,T)$.

   Let us recast System~\eqref{eq:LWR-sd} as a fixed point equation, as 
    \begin{equation*}
\text{for every $t\in(0,T)$,} \quad \vr_n(t) - \vr_0 = \int_0^t g(\vr_n(s),\vu_n(s))\, \der s,
    \end{equation*}
we obtain by letting $n$ go to $+\infty$, 
    \begin{equation*}
        \vr_*(t) - \vr_0 = \int_0^t g(\vr_*(s),\vu_*(s))\, \der s,
    \end{equation*}
meaning that $\vr^*$ solves System~\eqref{eq:LWR-sd} with $\vu^*$ as control. Finally, according to the aforementioned convergences and since the functionals $S$ and $B$ are convex, it is standard that one has 
$$
\lim_{n\to +\infty}C_T(\vu_n)=C_T(\vu^*),\quad \liminf_{n\to +\infty}S(\vu_n)\geq S(\vu^*), \quad \liminf_{n\to +\infty}B(\vu_n)\geq B(\vu^*).
$$
We get that $\J(\vu^*)\leq \liminf_{n\to +\infty} \J(\vu_n)$, and we infer that $\vu^*$ solves Problem~\eqref{eq:control}. 
\end{proof}

\subsection{Optimality conditions}\label{sec:optimCond}

We have established the existence of an optimal control in \Cref{theo:existOCP}. In order to derive a numerical solution algorithm, we will now state the necessary optimality conditions on which the algorithm we will build is based. One of the difficulties is that the functional we use involves non-differentiable quantities. We will therefore first use the notion of subdifferential to write the optimality conditions. In Section~\ref{sec:num_method} dedicated to numerical experiments, we will explain how we approximate these quantities.

Let us first compute the differential of the functional $C_T$. To this aim, we introduce the tangent cone to the set $\U_{ad}$.

\begin{definition}\label{def:admPert}
Let $\vu\in  \U_{ad} $. A function $\vh$ in $L^{\infty}(0,T)$ is said to be an ``admissible perturbation'' of $\vu$ in $\U_{ad}$ if, for every sequence of positive real numbers $(\varepsilon_n)_{n\in\N}$ decreasing to 0, there exists a sequence of functions $\vh^n$ converging to $\vh$ for the weak-star topology of $L^{\infty}(0,T)$ as $n \rightarrow +\infty$, and such that $\vu+\varepsilon_n \vh^n \in \U_{ad}$ for every $n \in \mathbb{N}$.
\end{definition}

\begin{proposition}\label{prop:derCT}
Let $\vu\in\U_{ad}$ and $(\vr,\vg)$ the associated solution to \eqref{eq:LWR-sd}. We introduce the two matrices $M$ and $N$ defined from the Jacobian matrices of $f^{FV}$ and $\phi^{LP}$ as
	\begin{align}\label{eq:MNdef}
M(\vr,\vg,\vu) &= (\d_\vr f^{FV})(\vr,\vg) + (\d_\vg f^{FV})(\vr,\vg) \, (\d_\vr\phi^{LP})(\vr,\vu),\\
			N(\vr,\vg,\vu) &= (\d_\vg f^{FV})(\vr,\vg)(\d_\vu\phi^{LP})(\vr,\vu),
	\end{align}
where we use the notational conventions introduced in Section~\ref{sec:intro}.

The functional $C_T$ is differentiable in the sense of G\^ateaux and its differential reads
	\begin{equation}\label{eq:dC}
\der C_T(\vu)\vh = \int_0^T (N(\vr,\vg,\vu)^\top  \vp)\cdot \vh \, \der t,
	\end{equation}
for every admissible perturbation $\vh$, where $\vp$ is the so-called adjoint state, defined as the unique solution to the Cauchy system
	\begin{equation}\label{eq:dual}
		\begin{cases}
			\vp' + M(\vr,\vg,\vu)^\top  \vp = 0\quad \text{in }(0,T),\\
			\vp(T) = \vc.
		\end{cases}
	\end{equation}
\end{proposition}
\begin{remark}
In \Cref{prop:derCT} above, the matrices $M$ and $N$ express respectively the way by which $\vr$ interacts with $\vg$ and $\vg$ interacts with $\vu$.
\end{remark}
\begin{proof}[Proof of~\Cref{prop:derCT}]
Let $\vu\in \U_{ad}$. The G\^ateaux differentiability of $C_T$, $\U_{ad}\ni \vu \mapsto \vr$ and $\U_{ad}\ni \vu \mapsto \vg$ is standard, and follows for instance directly of the proof of the Pontryagin Maximum Principle (PMP, see e.g. \cite{Lee-Markus}). Although the expression of the differential of $C_T$ could also be obtained by using the PMP, we provide a short proof hereafter to make this article self-contained.
 
Let $\vh\in\Linf$ be an admissible perturbation of $\vu$ in $\U_{ad}$. One has
\begin{equation}\label{eq:dC1}
    \der C_T(\vu)\,\vh = \vc\cdot\pt{\vr}(T),
\end{equation}
where $\pt{\vr} = \der(\vu\mapsto\vr)\vh$ solves the system  
\begin{equation*} 
	\begin{cases}
		\dot\vr' = (\d_\vr f^{FV})(\vr,\vg)\dot\vr + (\d_\vg f^{FV})(\vr,\vg)\dot\vg, \\
		\dot\vg = (\d_\vr\phi^{LP})(\vr,\vu)\dot\vr + (\d_\vu\phi^{LP})(\vr,\vu)\vh, \\
		\dot\vr(0) = 0.
	\end{cases}
\end{equation*}
We infer that $\dot \vr$ solves
\begin{equation}\label{eq:dotvr}
	\begin{cases}
		\dot\vr' = M(\vr,\vg,\vu) \dot\vr + N(\vr,\vg,\vu) \vh, \\
		\dot\vr(0) = 0,
	\end{cases} 
\end{equation}
with $M$ and $N$ as defined in \eqref{eq:MNdef}.
Let us multiply the main equation of \eqref{eq:dual} by $\dot \vr$ in the sense of the inner product, and integrate over $(0,T)$. We obtain:
$$
\int_0^T \dot\vr \cdot \frac{\der \vp}{\der t} \, \der t + \int_0^T \dot\vr\cdot(M(\vr,\vg,\vu)^\top  \vp) \, \der t=0.
$$
Similarly, let us multiply the main equation of \eqref{eq:dotvr} by $\vp$ in the sense of the inner product, and integrate over $(0,T)$. We obtain:
$$
\int_0^T \vp \cdot \frac{\der \dot\vr}{\der t} \, \der t - \int_0^T M(\vr,\vg,\vu)\dot\vr\cdot \vp \, \der t=\int_0^T N(\vr,\vg,\vu)\vh\cdot \vp\, \der t.
$$
Summing the two last equalities above yields
$$
\dot\vr(T)\cdot \vp(T)-\dot\vr(0)\cdot \vp(0)=\int_0^T \vh\cdot (N(\vr,\vg,\vu)^\top \vp)\, \der t.
$$
Using this identity with $\vp(T) = c$ and $\dot\vr(0) = 0$ results in Expression \eqref{eq:dC}.
\end{proof}

From this result, we will now state the optimality conditions for Problem~\eqref{eq:control}. Let us first recall that, according to \cite[Proposition~I.5.1]{MR1727362}, the subdifferential of the total variation is well-known, given by
$$
\partial \operatorname{TV}(\vu^*)=\{\boldsymbol{\eta} \in C^0([0,T];\R^{N_r})\mid \Vert \boldsymbol{\eta} \Vert_\infty\leq 1\text{ and }\int \boldsymbol{\eta} {\rm d}\vu^*=\operatorname{TV}(\vu^*)\}.
$$

Let us denote by $e_i$ the $i$-th vector of the canonical basis of $\R^{N_r}$
\begin{theorem}\label{theo:copt1332}
Let $\vu^*=(u_i^*)_{1\leq i\leq N_r}$,  denote a solution to Problem~\eqref{eq:control}, $(\vr,\vg)$ the associated solution to \eqref{eq:LWR-sd}. and let $i\in \llbracket 1,N_r\rrbracket$. There exists $\boldsymbol{\eta^*}\in \partial \operatorname{TV}(\vu^*)$ such that
\begin{itemize}
\item on $\{u^*_i=0\}$, one has $\Psi\cdot e_i \geq 0$,
\item on $\{u^*_i=1\}$, one has $\Psi\cdot e_i \leq 0$,
\item on $\{0<u^*_i<1\}$, one has $\Psi\cdot e_i = 0$,
\end{itemize}
where the function $\Psi:[0,T]\to \R^{N_r}$ is given by
$$
\Psi(t)=N(\vr(t),\vg(t),\vu^*(t))^\top \vp^*(t)+\theta_S\left(\sum_{i=1}^{N_r}u_i^*(t)-N_{\max}\right)_++\theta_B\boldsymbol{\eta^*}(t)
$$
and where $\vp^*$ denotes the adjoint state introduced in Proposition~\ref{prop:derCT}, associated to the control choice $\vu^*$.
\end{theorem}
\begin{remark}
Written in this way, the first order optimality conditions are difficult to use. In the next section, we will introduce an approximation of the total variation of $\vu^*$ leading to optimality conditions more easily usable within an algorithm.
\end{remark}

\begin{proof}[Proof of Theorem~\ref{theo:copt1332}]

To derive the first order optimality conditions for this problem, it is convenient to introduce the so-called indicator function $\iota_{\U_{ad}}$ given by
$$
\iota_{\U_{ad}}(\vu)=\left\{\begin{array}{ll}
0 & \text{if }\vu\in \U_{ad}\\
+\infty & \text{else}.
\end{array}\right.
$$
Observing that the optimization problem we want to deal with can be recast as
$$
\min_{\vu\in L^\infty((0,T);\R^{N_r})}\mathcal{J}_\theta(\vu)+\iota_{\U_{ad}}(\vu),
$$
it is standard in nonsmooth analysis to write the first order optimality conditions as: 
$$
0\in  \partial \left(\mathcal{J}_\theta(\vu^*)+ \iota_{\U_{ad}}(\vu^*)\right),
$$
or similarly, by using standard computational rules \cite{MR1727362},
$$
-\partial C_T(\vu^*)-\frac{\theta_S}{2}\partial S(\vu^*)\in  \theta_B\partial B(\vu^*)+\partial \iota_{\U_{ad}}(\vu^*),
$$
Let $\vu\in \U_{ad}$. The condition above yields the existence of $\boldsymbol{\eta^*}=(\eta_i^*)_{1\leq i\leq N_r}\in \partial \operatorname{TV}(\vu^*)$ such that
$$
\der C_T(\vu^*)(\vu-\vu^*) +\theta_S\der S(\vu^*)(\vu-\vu^*)+\theta_B\sum_{i=1}^{N_r}\langle \eta_i^*,u_i-u_i^*\rangle_{L^2(0,T)}\geq 0.
$$
Since $\vu$ is arbitrary, we infer that for any admissible perturbation  $\vh$ of $\vu^*$ (see Definition~\ref{def:admPert}), one has
$$
\der C_T(\vu^*)\vh +\theta_S\der S(\vu^*)\vh+\theta_B\sum_{i=1}^{N_r}\langle \eta_i^*,h_i\rangle_{L^2(0,T)}\geq 0,
$$
or similarly
\begin{equation}\label{strasb:1750}
\int_0^T \vh\cdot \left( N(\vr,\vg,\vu) ^\top \vp^*+\theta_S\left(\sum_{i=1}^{N_r}u_i^*-N_{\max}\right)_++\theta_B\boldsymbol{\eta^*}\right)\, \der t\geq 0.
\end{equation}
To analyze this optimality condition, let us introduce the function $\Psi:[0,T]\to \R^{N_r}$ defined by
\begin{equation}\label{def:Psi}
\Psi(t)=N(\vr,\vg,\vu) ^\top \vp^*(t)+\theta_S\left(\sum_{i=1}^{N_r}u_i^*-N_{\max}\right)_++\theta_B\boldsymbol{\eta^*}(t).
\end{equation}

 In what follows, we will write the optimality conditions holding for the $i$-th component of $\vu^*$, where $i\in \llbracket 1,N_r\rrbracket$ is given. 
 
 Let us assume that the set $\mathcal{I}=\{0<u_i^*<1\}$ is of positive Lebesgue measure. Let $x_0$ be a Lebesgue point of $u_i^*$ in $\mathcal{I}$ and let $(G_{n})_{n\in \N}$ be a sequence of measurable subsets with $G_{n}$ included in $\mathcal{I}$ and containing $x_0$. Let us consider $\vh=(h_j)_{1\leq j\leq N_r}$ such that $h_j=0$ for all $j\in \llbracket 1,N_r\rrbracket\backslash \{i\}$ and $h_i=\mathds{1}_{G_{n}}$. Notice that $\vu^* \pm\eta \vh$ belongs to $\mathcal{U}_{ad}$ whenever $\eta$ is small enough. According to \eqref{strasb:1750}, one has
\[
\pm \int_{G_n} \Psi(t)\cdot e_i\, \der t\geq 0.
\] 
Dividing this inequality by $|G_{n}|$ and letting $G_{n}$ shrink to $\{x_0\}$ as $n\to +\infty$ shows that one has
$$
\Psi(t)\cdot e_i=0, \quad \text{a.e. in }\mathcal{I}.
$$

Let us now assume that the set $\mathcal{I}_1=\{u_i^*=1\}$ is of positive Lebesgue measure. Then, by mimicking the reasoning above, we consider $x_1$, a Lebesgue point of $u_i^*$ in $\mathcal{I}_1$, and $\vh=(h_j)_{1\leq j\leq N_r}$ such that $h_j=0$ for all $j\in \llbracket 1,N_r\rrbracket\backslash \{i\}$ and $h_i=-\mathds{1}_{G_{n}}$, where $(G_{n})_{n\in \N}$ is a sequence of measurable subsets with $G_{n}$ included in $\mathcal{I}_1$ and containing $x_1$. According to \eqref{strasb:1750}, one has
\[
- \int_{G_n} \Psi(t)\cdot e_i\, \der t\geq 0.
\] 
As above, we divide this inequality by $|G_{n}|$ and let $G_{n}$ shrink to $\{x_1\}$ as $n\to +\infty$. We recover that $\Psi(t)\cdot e_i\leq 0$.

Regarding now the set $\mathcal{I}_0=\{u_i^*=0\}$, the reasoning is a direct adaptation of the case above, which concludes the proof.
\end{proof}

\section{Towards a numerical algorithm}
\label{sec:num_method}
In this section, we introduce an exploitable approximation of the problem we are dealing with and describe the algorithm implemented in the numerical part.
 
\subsection{An approximate version of Problem~\eqref{eq:control}}\label{sec:approxContversion}
The fact that Problem~\eqref{eq:control} involves the total variation of control makes the problem non-differentiable. Of course, dedicated algorithms exist to take into account such a term in the solution, for instance Chambolle’s projection algorithm \cite{chambolle2004algorithm}. Nevertheless, in order to avoid too costly numerical approaches, we have chosen to consider a simple differentiable approximation of the term $B(\vu)$, namely 
\begin{equation}\label{eq:bv}
    B_\nu(\vu) = \sum_{i=1}^{N_r}\operatorname{TV_\nu}(u_i),
\end{equation}
where $\nu>0$ is a small parameter and the total variation $\operatorname{TV}(u_i)$ is approximated by a differentiable functional in $L^2(0,T)$, denoted $\operatorname{TV}_\nu(u_i) $, where $\nu>0$ stands for a regularization parameter.
The concrete choice of the differentiable approximation of the TV standard will be discussed in the rest of this section. We will also give some elements on its practical implementation.


We are thus led to consider the following approximate version of Problem~\eqref{eq:control}:
\begin{equation}\label{eq:control_approx} \tag{$\mathcal{P}_{\theta,\nu}$}
\boxed{\inf_{\vu \in \U_{ad}}{\mathcal{J}_{\theta,\nu}(\vu)}}.
\end{equation}
where $\U_{ad}$ is given by \eqref{def:Uad}, and $\mathcal{J}_{\theta,\nu}$ is given by
\begin{equation}
\mathcal{J}_{\theta,\nu}(\vu) = C_T(\vu) + \frac{\theta_S}{2} S(\vu) + \theta_B B_\nu(\vu),
\end{equation}
We will see that this approximation is in fact well adapted to a practical use. Indeed, the first order optimality conditions for this approximated problem can be rewritten in a very concise and workable way, unlike the result stated in Theorem~\ref{theo:copt1332}. This is the purpose of the following result.  

\begin{theorem}\label{theo:2301}
Let $\vu^*\in \U_{ad}$ denote a local minimizer for Problem~\eqref{eq:control_approx} and $(\vr,\vg)$ the associated solution to \eqref{eq:LWR-sd}. Then, $\vu^*$ satisfies the first order necessary condition
$$
\L(\cdot)=0\quad \text{a.e. on }[0,T],
$$
where $\Lambda:[0,T]\to \R^{N_r}$ is defined by
    \[ 
        \L(t) = \min\left\{\vu^*(t),\max\left\{\vu^*(t)-1,\nabla_{\vu} \mathcal{J}_{\theta,\nu}(\vu^*)(t)\right\}\right\},
    \]
    and 
    $$
    \nabla_{\vu} \mathcal{J}_{\theta,\nu}(\vu^*):[0,T]\ni t \mapsto N(\vr(t),\vg(t),\vu(t)) ^\top \vp^*(t)+\theta_S\left(\sum_{i=1}^{N_r}u_i^*(t)-N_{\max}\right)_++\theta_B\sum_{i=1}^{N_r}\nabla_{\vu}\operatorname{TV}_\nu(u_i)(t),
    $$
where $\vp^*$ has been introduced in Theorem~\ref{theo:copt1332}, the min, max operations being understood componentwise, and the term $\nabla_{\vu}$ denoting the gradient with respect to $\vu$ in $L^2(0,T)$.
\end{theorem}

\begin{proof}
 The proof is similar to the proof of Theorem~\ref{theo:copt1332}. Indeed, let $\vu^*$ be a local minimizer for Problem~\eqref{eq:control_approx}. The first order optimality conditions are given by the so-called Euler inequation and read
 $
 \der \mathcal{J}_{\theta,\nu}(\vu^*)\vh\geq 0,
 $
 or similarly
 $$
 \int_0^T \nabla_\vu \mathcal{J}_{\theta,\nu}(\vu^*)\cdot \vh\, \der t\geq 0.
 $$
 for every admissible perturbation $\vh$, as defined in Definition~\ref{def:admPert}. Let us fix $i\in \llbracket 1,N_r\rrbracket$. By mimicking the reasoning involving Lebesgue points in the proof of Theorem~\ref{theo:copt1332}, we get
 \begin{itemize}
\item on $\{u^*_i=0\}$, one has $\nabla_{\vu}\mathcal{J}_{\theta,\nu}(\vu^*)\cdot e_i \geq 0$,
\item on $\{u^*_i=1\}$, one has $\nabla_{\vu}\mathcal{J}_{\theta,\nu}(\vu^*)\cdot e_i \leq 0$,
\item on $\{0<u^*_i<1\}$, one has $\nabla_{\vu}\mathcal{J}_{\theta,\nu}(\vu^*)\cdot e_i = 0$.
\end{itemize}
Note that, on $\{u^*_i=0\}$, one has $\nabla_\vu \mathcal{J}_{\theta,\nu}(\vu^*)\cdot e_i\geq 0$ and then 
        $\L(t) \cdot e_i=  \min\{0, \max\{-1, \nabla_\vu \mathcal{J}_{\theta,\nu}(\vu^*)(t)\cdot e_i\}\} = 0$. On $\{u^*_i=1\}$,  one has $\nabla_\vu \mathcal{J}_{\theta,\nu}(\vu^*)\cdot e_i\leq 0$ and therefore 
        $\L(t) \cdot e_i=  \min\{1, \max\{0, \nabla_\vu \mathcal{J}_{\theta,\nu}(\vu^*)(t)\cdot e_i\}\} = 0$. On $\{0<u^*_i<1\}$,  one has $\nabla_\vu \mathcal{J}_{\theta,\nu}(\vu^*)\cdot e_i= 0$ and therefore 
        $\L(t) \cdot e_i=  \min\{u_i^*, \max\{u_i^*(t)-1, 0\}\} = 0$.

Conversely, let us assume that $\Lambda(\cdot)=0$. On $\{u^*_i=0\}$, one has 
$$
0=\L(t) \cdot e_i=\min\{0, \max\{-1,\nabla_\vu \mathcal{J}_{\theta,\nu}(\vu^*)(t)\cdot e_i\}\},
$$ 
so that $\max\{-1,\nabla_\vu \mathcal{J}_{\theta,\nu}(\vu^*)(t)\cdot e_i\}\geq 0$ and finally, $\nabla_\vu \mathcal{J}_{\theta,\nu}(\vu^*)(t)\cdot e_i \geq 0$. A similar reasoning yields the optimality conditions on $\{u^*_i=1\}$ and $\{0<u^*_i<1\}$. The conclusion follows.
\end{proof}

\paragraph{Practical computation of the TV operator gradient in a discrete framework.}

From a practical point of view, we discretize the space of controls, which leads us to consider a time discretization denoted $(t_n)_{1\leq n\leq N_T}$ with $N_T\in \N^*$ fixed, as well as piecewise constant controls in time, denoted  $(u_i^n)_{\substack{1\leq i\leq N_r\\ 1\leq n\leq N_T}}$.

Hence, the term $u_i^n$ corresponds to the control at road $i$ and time $n$. The simplest discrete version of the TV semi-norm is given by 
\begin{equation*}
    \operatorname{TV}(\vu) \simeq  \sum_{i=1}^{N_r}\sum_{n=2}^{N_T}|u_i^n - u_i^{n-1}|.
\end{equation*}
In order to manipulate differentiable expressions, we introduce for some $\nu>0$, the smoothed discrete TV operator is given by 
\begin{equation}\label{eq:TVreg}
    \operatorname{TV}_\nu: \R^{N_r\times N_T}\ni (u_i^n)_{\substack{1\leq i\leq N_r\\ 1\leq n\leq N_T}} \mapsto \sum_{i=1}^{N_r}\sum_{n=2}^{N_T}a_\nu(u_i^n - u_i^{n-1}),
\end{equation}
where $ a_\nu:\R\ni x \mapsto \sqrt{x^2 + \nu^2}$.

Let $(i_0,n_0)\in \llbracket 1,N_r\rrbracket\times \llbracket 1,N_T\rrbracket$. In what follows, we will use a discrete equivalent of Theorem~\ref{theo:2301}, involving the gradient of $ \operatorname{TV}_\nu$, obtained from the expression
$$
\partial_{u_{i_0}^{n_0}}  \operatorname{TV}_\nu(\vu)=\left\{
\begin{array}{ll}
a_\nu'(u_{i_0}^{n_0}-u_{i_0}^{n_0-1})-a_\nu'(u_{i_0}^{n_0+1}-u_{i_0}^{n_0}) & \text{if }2\leq n_0\leq N_T-1\\
a_\nu'(u_{i_0}^{N_T}-u_{i_0}^{N_T-1}) & \text{if }n_0=N_T\\
-a_\nu'(u_{i_0}^{2}-u_{i_0}^{1}) & \text{if }n_0=1,
\end{array}
\right.
$$ 
where $\vu=(u_i^n)_{\substack{1\leq i\leq N_r\\ 1\leq n\leq N_T}}$ is given

%

\subsection{Numerical solving of the primal and dual problems}

The models we use are already discretized in space. We now explain how we discretize them in time.
We recall that, at the end of section~\ref{sec:approxContversion}, we have already considered that the controls are assimilated to piecewise constant functions on each cell of the considered mesh and on each time step.  

\paragraph{Finite volumes scheme for the primal problem.}

The main transport equation on network is solved by integrating for each road the finite volume flow given by \eqref{eq:LWR-sd} with an explicit Euler scheme: 
\begin{equation*}
    \r_{i,j}^{n+1} = \r_{i,j}^{n} - \frac{\D t_n}{\D x_{i,j}}\left(\F_{i,j+\frac{1}{2}}^{n} - \F_{i,j-\frac{1}{2}}^{n}\right), \quad 1\le i\le N_r,\quad 1\leq j\leq N_c,
\end{equation*}
using the local-Lax numerical flux 
\begin{equation*}
	\F_{i,j+\frac{1}{2}}^{n} = \F(\r_{i,j}^{n}, \r_{i,j+1}^{n}),\quad 2\leq j\leq N_c-1,
\end{equation*}
and the Neumann boundary conditions
$$
    \F_{i,\frac{1}{2}}^n   := \g_{i,L}^n, \quad \F_{i,N_c+\frac{1}{2}}^n := \g_{i,R}^n, \quad 1\le i\le N_r,
$$
where $ \g_{i,L}^n$ and $ \g_{i,R}^n$ are obtained as  the solution to the linear programming system
\begin{equation*}
	\vg^n = \phi^{\text{LP}}(\vr^{n}, \vu^{n}).
\end{equation*}

\paragraph{Euler scheme for adjoint problem.}
To solve the backward ODE \eqref{eq:dual},it is convenient to introduce $Z(t) := (M(T-t))^T$, and $\vq(t) := \vp(T-t)$ such that we are now dealing with the  Cauchy system:
\begin{equation*}
    \begin{cases}
        \vq'(t) 
        = Z(t)\,\vq(t), & t\in(0,T), \\ 
        \vq(0) = \vc, & \,
    \end{cases}
\end{equation*} 
that we integrate with a classical explicit Euler scheme:
$$
    \vq^1 = \vc,\quad \vq^{n+1} = \left(I+\D t_n\,Z^n\right) \vq^n, \quad n=1,\cdots,N_T-1.
$$
The solution is finally recovered using that $\vp^n = \vq^{N_T-n+1}$.

\subsection{Optimization algorithms}

The starting point of the algorithm we implement is based on a standard primal-dual approach, in which the state and the adjoint are computed in order to deduce the gradient of the considered functional. We combine it with a projection method in order to guarantee the respect of the $L^\infty$ constraints on the control. This method has the advantage of being robust, as it generally allows a significant decrease of the cost functional. On the other hand, it often has the disadvantage of being very local, which results in an important dependence on the initialization. Moreover, one can expect that there are many local minima, since the targeted problem is intrinsically of infinite dimension, which makes the search difficult.

We will propose a modification of this well-known method, using a fixed point method inspired by the optimality condition stated in Theorem~\ref{theo:2301}.

\subsubsection{Projected gradient descent}

A direct approach is to consider the gradient algorithm, in which we deal with the condition $0\leq u^k_i\leq 1$ by projection, according to Algorithm~\ref{alg:GD}.
The specific difficulty of this approach is to find a suitable descent-step $\delta_k$, which must be small enough to ensure descent but large enough for the algorithm to converge in a reasonable number of iterations. We hereby combine this algorithm with a scheduler \eqref{eq:scheduler} inspired from classical learning rate scheduler in deep-learning \cite{wu_demystifying_2019} to select an acceptable descent step, where $\delta_0$ and $\operatorname{decay}$ are given real numbers.

\begin{equation}\label{eq:scheduler}
    \delta_k := \frac{\delta_0}{1 + \operatorname{decay}\times k}.
\end{equation}

\begin{algorithm}
	\caption{Optimal control by projected gradient descent method (GD)}
	\begin{algorithmic}\label{alg:GD}
		\REQUIRE $\vr^0, \vu^0$, $\operatorname{tol}>0,\operatorname{itermax} > 0$
		\STATE \textbf{Initialization: } $k=0$
		\WHILE{$||\vu^{k+1}-\vu^k|| > \operatorname{tol}$ \AND $k \leq\operatorname{itermax}$}
			\STATE $(\vr, \vg)\leftarrow$ solution to $\begin{cases} 
				\vr'= f^{\text{FV}}(\vr,\vg) \\ 
				\vg = \phi^{\text{LP}}(\vr,\vu^{k}) \\ 
				\vr(0) = \vr^0
			\end{cases}$
			\STATE $\vp \leftarrow$ solution to $\begin{cases}
				\vp'+M(\vr,\vg,\vu^k)^T\vp=\bf{0} \\ 
				\vp(T) = \vc
			\end{cases}$
			\STATE $\vu^{k+1}\leftarrow \operatorname{proj}_{[0,1]}(\vu^k - \delta_k\nabla \mathcal{J}_{\theta,\nu}(\vu^k))$, with $\delta_k>0$ such that $\mathcal{J}_{\theta,\nu}(\vu^{k+1})\leq \mathcal{J}_{\theta,\nu}(\vu^k)$
			\STATE $k \leftarrow k+1$
    \ENDWHILE
  \end{algorithmic}
\end{algorithm}



\paragraph{Fixed point method}

A fixed-point (FP) algorithm is derived using the first-order optimality conditions stated in Theorems~\ref{theo:copt1332} and \ref{theo:2301}, by rewriting them in a fixed-point formulation. We have seen that they write under the form 
\[
	u_i \in I_\mu \Rightarrow \d_{u_i}\mathcal{J}_{\theta,\nu}(\vu)\in E_\mu,
\]
where $\mu \in  \{0, *, 1\}$, $I_0 = \{0\}$, $I_*=(0,1)$, $I_1=\{1\}$, $E_0 = \R_+$,  $E_*=\{0\}$, $E_1=\R_-$. 
This rewrites as
\[
	u_i \in I_\mu \Rightarrow u_i \in F_\mu:=\{u_i|\d_{u_i}\mathcal{J}_{\theta,\nu}(\vu)\in E_\mu\}.
\]
This leads us to compute $u_i$ by using the following fixed-point relationship
\begin{equation}
	u_i^{k+1} = 1 \times \mathds{1}_{\{\d_{u_i}\mathcal{J}_{\theta,\nu}(\vu^k) < -\kappa\}} + u_i^k\,\mathds{1}_{\{ -\kappa \leq \d_{u_i}\mathcal{J}_{\theta,\nu}(\vu^k) \leq \kappa\}} + 0 \times \mathds{1}_{\{\d_{u_i}\mathcal{J}_{\theta,\nu}(\vu^k) > \kappa\}},
\end{equation}
where $\kappa>0$ is a small threshold preventing too small gradients from modifying the control.
This finally leads us to Algorithm~\ref{alg:FP}. 
\begin{remark}\label{rk:fixed}
    The advantage of the fixed-point formulation can be illustrated by the fact that some common control configurations can lead to a flat gradient while optimisation is still possible, i.e. gradient descent is too local to detect certain appropriate (and existing) descent direction. Let's construct such a case where the fixed-point algorithm has a major advantage over gradient descent. We recall the constraints of the linear programming for the $2 \times 2$ case: 

\begin{eqnarray}
    0 & \leq & \g_1^R\leq\g_1^{R,\max}, \label{eq:constraint1} \\ 
    0 & \leq & \g_2^R\leq\g_2^{R,\max}, \label{eq:constraint2} \\
    0 & \leq & \au\g_1^R + \bu\g_2^R \leq (1-u_3)\g_3^{L,\max}, \label{eq:constraint3} \\
    0 & \leq & (1-\au)\g_1^R + (1-\bu)\g_2^R \leq (1-u_4)\g_4^{L,\max}. \label{eq:constraint4}
\end{eqnarray}
Considering for instance a case where $u_3=1$, then \eqref{eq:constraint3} yields $\au\g_1^R + \bu\g_2^R = 0$. By positivity, this implies that $\au\g_1^R = \bu\g_2^R=0$. If furthermore $u_4 > 0$, then it follows that $\au=P_\e^{\bar{\a}}(1-u_4)\neq 0$ and $\bu=P_\e^{\bar{\b}}(1-u_4)\neq 0$ (see~\eqref{eq:poly} in Appendix for the definition of these functions). It remains that necessarly, $\g_1^R=\g_2^R=0$ and the junction is blocked.
The point here is that the only way to influence the network -- therefore acting on the value of $\mathcal{J}_{\theta,\nu}$ -- is by having a control such that $\g_1^R$ or $\g_2^R$ becomes positive, and this can be obtained only when $u_4$ is set to 0. Indeed, we would now have $\au=P_\e^{\bar{\a}}(1)=P_\e^{\bar{\b}}(1)=\bu=0$, releasing the constraint given by \eqref{eq:constraint3} on the values of $\g_1^R, \g_2^R$.

However, if the algorithm finds a perturbation $\delta u_4>0$ small enough such that $u_4-\delta u_4\neq 0$, it would yields $\mathcal{J}_{\theta,\nu}(\vu) = \mathcal{J}_{\theta,\nu}(\vu - \delta u_4 e_4)$, thus $\d_{u_4}\mathcal{J}_{\theta,\nu}(\vu)=0$. As a result, the gradient descent step is stationnary for $u_4$ even though the perturbation $\delta u_4$ was in the good direction in order to decrease $\mathcal{J}_{\theta,\nu}$ after several more iterations. 

With the fixed-point algorithm instead, we find the suitable control $u_4$ in one iteration from the previous configuration thanks to the indicator functions, since the sign of the gradient contains all the information that we needed. This allows us to circumvent the threshold phenomenon of the gradient described above.
\end{remark}

\begin{algorithm}[H]
    \caption{Optimal control with Fixed Point method (FP)}
    \begin{algorithmic}\label{alg:FP}
        \REQUIRE $\vr^0,\vu^0$, $\operatorname{tol}>0,\operatorname{itermax}> 0$
        \STATE \textbf{Initialization: } $k=0$
        \WHILE{$||\vu^{k+1}-\vu^k|| > \operatorname{tol}$ \AND $k \leq\operatorname{itermax}$}
            \STATE $(\vr, \vg)\leftarrow$ solution to $\begin{cases} 
                \vr'= f^{\text{FV}}(\vr,\vg) \\ 
                \vg = \phi^{\text{LP}}(\vr,\vu^{k}) \\ 
                \vr(0) = \vr^0
            \end{cases}$
            \STATE $\vp \leftarrow$ solution to $\begin{cases}
                \vp'+M(\vr,\vg,\vu^k)^T\vp=\bf{0} \\ 
                \vp(T) = \vc
            \end{cases}$
            \STATE $\vu^{k+1}\leftarrow \mathds{1}_{\{\nabla \mathcal{J}_{\theta,\nu}(\vu^k) < -\kappa\}} + \vu^k\mathds{1}_{\{ -\kappa \leq \nabla {\mathcal{J}_{\theta,\nu}(\vu^k)} \leq \kappa\}} + 0 \times \mathds{1}_{\{\d_{u_i}\mathcal{J}_{\theta,\nu}(\vu^k) > \kappa\}}$
            \STATE $k \leftarrow k+1$
    \ENDWHILE
    \end{algorithmic}
\end{algorithm}

It is noteworthy that the convergence of a fixed-point algorithm applied to an extremal eigenvalue problem has been established in \cite{chambolle2023stability}, for certain parameters regimes involved in the problem. To the best of our knowledge, there is no general convergence result for this type of optimal control problem. 

\subsubsection{Hybrid GDFP methods}
Since gradient descent theoretically always guarantees a direction of descent (provided we choose a small enough step) but can be slow and/or remain trapped in a local extremum (see Remark~\ref{rk:fixed}), and since the fixed point appears more exploratory but does not guarantee descent, it seems worthwhile to investigate the hybridization of both methods. 

The chosen algorithm is implemented by computing most of the iterations by gradient descent and using the fixed point method every $K\in\N$ iterations in the expectation of escaping from possible basins of attraction of local minimizers. Note that $K$ is a hyper-parameter of the method. This is summarized in Algorithm~\ref{alg:GDFP}. 

Instead of performing FP steps regularly, we can choose to space them in order to have more more iterations for  the gradient descent to converge. This second version is given in  Algorithm~\ref{alg:FP_trigger} where the time between two FP steps increases by a factor of $\tau$.

\begin{algorithm}
	\caption{Optimal control with hybrid GD \& FP algorithm (GDFP)}
	\begin{algorithmic}\label{alg:GDFP}
		\REQUIRE $\vr^0,\vu^0$, $\operatorname{tol}>0,\operatorname{itermax} > 0$
		\STATE \textbf{Initialization: } $k=0$
		\WHILE {$||\vu^{k+1}-\vu^k|| > \operatorname{tol}$ \AND $k \leq\operatorname{itermax}$}
			\STATE $(\vr, \vg)\leftarrow$ solution to $\begin{cases} 
					\vr'= f^{\text{FV}}(\vr,\vg) \\ 
					\vg = \phi^{\text{LP}}(\vr,\vu^{k}) \\ 
					\vr(0) = \vr^0
				\end{cases}$
			\STATE $\vp \leftarrow$ solution to $\begin{cases}
				\vp'+M(\vr,\vg,\vu^k)^T\vp=\bf{0} \\ 
				\vp(T) = \vc
			\end{cases}$
			\IF {$k\% K == 0$} 
				\STATE $\vu^{k+1}\leftarrow \mathds{1}_{\{\nabla \mathcal{J}_{\theta,\nu}(\vu^k) < -\kappa\}} + \vu^k\mathds{1}_{\{ -\kappa \leq \nabla {\mathcal{J}_{\theta,\nu}(\vu^k)} \leq \kappa\}} + 0 \times \mathds{1}_{\{\d_{u_i}\mathcal{J}_{\theta,\nu}(\vu^k) > \kappa\}}$
			\ELSE
			\STATE $\vu^{k+1}\leftarrow \operatorname{proj}_{[0,1]}(\vu^k - \delta_k\nabla \mathcal{J}_{\theta,\nu}(\vu^k))$, with $\delta_k>0$ such that $\mathcal{J}_{\theta,\nu}(\vu^{k+1})\leq \mathcal{J}_{\theta,\nu}(\vu^k)$
			\ENDIF
			\STATE $k \leftarrow k+1$
	\ENDWHILE
  \end{algorithmic}
\end{algorithm}

\begin{algorithm}
	\caption{Optimal control with GDFP with spaced FP steps}
	\begin{algorithmic}\label{alg:FP_trigger}
		\REQUIRE $\vr^0,\vu^0$, $tol>0,\text{itermax} > 0$, $K_0$, $\tau$
		\STATE \textbf{Initialization: } $k=0$, $K=K_0$
		\WHILE {$||\vu^{k+1}-\vu^k|| > tol$ \AND $k \leq \text{itermax}$}
			\STATE $(\vr, \vg)\leftarrow$ solution of $\begin{cases} 
					\vr'= f^{\text{FV}}(\vr,\vg) \\ 
					\vg = \phi^{\text{LP}}(\vr,\vu^{k}) \\ 
					\vr(0) = \vr^0
				\end{cases}$
			\STATE $\vp \leftarrow$ solution of $\begin{cases}
				\vp'(t)+M^T\vp=\bf{0} \\ 
				\vp(T) = \vc
			\end{cases}$
			\IF {$k\% K == 0$} 
				\STATE $\vu^{k+1}\leftarrow \mathds{1}_{\{\nabla \J(\vu^k) < -\kappa\}} + \vu^k\mathds{1}_{\{ -\kappa \leq \nabla {\J(\vu^k)} \leq \kappa\}} + 0 \times \mathds{1}_{\{\d_{u_i}\mathcal{J}_{\theta,\nu}(\vu^k) > \kappa\}}$
                \STATE $K \leftarrow \tau K$ 
			\ELSE
			\STATE $\vu^{k+1}\leftarrow \operatorname{proj}_{[0,1]}(\vu^k - \delta_k\nabla \J(\vu^k))$, with $\delta_k>0$ such that $\J(\vu^{k+1})\leq \J(\vu^k)$
			\ENDIF
			\STATE $k \leftarrow k+1$
	\ENDWHILE
  \end{algorithmic}
\end{algorithm}
\section{Numerical Results}
\label{sec:num_results}


This section presents some results obtained in various situations using the methods presented above. In all that follows, we assume that $\rho^{\max} = 1$,  $v^{\max} = 1$ and $L = 1$. We will first validate our approach on single junctions before studying more complex road networks. These are namely: a traffic circle and a three lanes network of intermediate size with a configuration unfavorable to our objective.



\subsection{Single junctions}

We start by  considering single junctions of type $1\! \times\! 1$, $1\! \times\! 2$, $2\! \times\! 1$ and $2\! \times\! 2$ as they will be the building blocks of the more complex networks.

We consider $50$ mesh cells per road and the initial density equals $0.66$. The route to empty is always composed of an incoming and an outgoing road. The time interval of the simulations is adjusted in order to allow the route to be completely emptied: the final time $T$ thus equals respectively $6$, $3.5$, $10$ and $5$ for the four junctions. We start the optimization algorithms with initial controls equal to $0$ and set the convergence threshold to having $\norm{\Lambda}$ less than $10^{-1}$, with a prescribed maximum of $100$ iterations. Neither constraints on the number of controls ($\theta_S=0$) nor BV regularization ($\theta_B=0$) are considered for these test cases.

Fig.~\ref{fig:single_junctions} shows us the comparison between the cost functional history when using the gradient descent (GD), the fixed-point (FP) and the hybrid (GDFP) methods. Specific numerical parameters of the algorithm are given in \eqref{tab:junctions}. We observe fundamental differences in behaviour between the cost functionals obtained by GD and the one obtained by FP. Indeed, GD allows a regular descent whereas FP generates jumps and oscillations, which allows a better exploration of the parameters and avoids certain unsatisfactory local minima. 
On Fig.~\ref{fig:single_junctions_lambda}, we also plotted the evolution of the optimality conditions $\Lambda$ for all methods: as expected, we observe that this quantity reaches values close to $0$ at the optimal point.

The obtained controls are shown in Fig.~\ref{fig:single_junctions_ctrl} and seem relevant. For instance, in the  $2\! \times\! 1$ case, the control of the outgoing road $3$ is fully activated only after time $t\approx 8$ to enable the emptying of the incoming road 1 first. We further note that the controls are essentially sparse and non-oscillating. They are almost bang-bang, i.e. take only values $0$ and $1$.

These results seem to suggest that the hybrid method is the most likely to generalize to larger graphs because of its ability to explore and find critical points while still being able to provide convergence. It is therefore the one we will use in the following.

\begin{table}[H]
    \centering
        \begin{tabular}{c|l|c|c|c|c} 
            \toprule 
            Symbol & Name & 1x1 & 1x2 & 2x1 & 2x2 \\ 
            \midrule
            K & FP trigger in GDFP & 3 & 5 & 10  & 2\\
            $\kappa$ & vanishing gradient threshold in FP & 0 & $10^{-10}$ & $0$ & $10^{-1}$ \\
            $\delta_0$ & initial descent step & 1 & $5\times 10^{-2}$ & $2\times 10^{-1}$ & $10^{-1}$ \\
            $\operatorname{decay}$ & decay in scheduler & $10^{-2}$ & $10^{-1}$ & $10^{-1}$ & $10^{-1}$ \\
            \bottomrule
         \end{tabular}
    \caption{Numerical parameters for single junctions.}
    \label{tab:junctions}
\end{table}




\begin{figure}[H]
    \centering
    \begin{subfigure}[t!]{0.4\textwidth}
        \includegraphics[width=\textwidth]{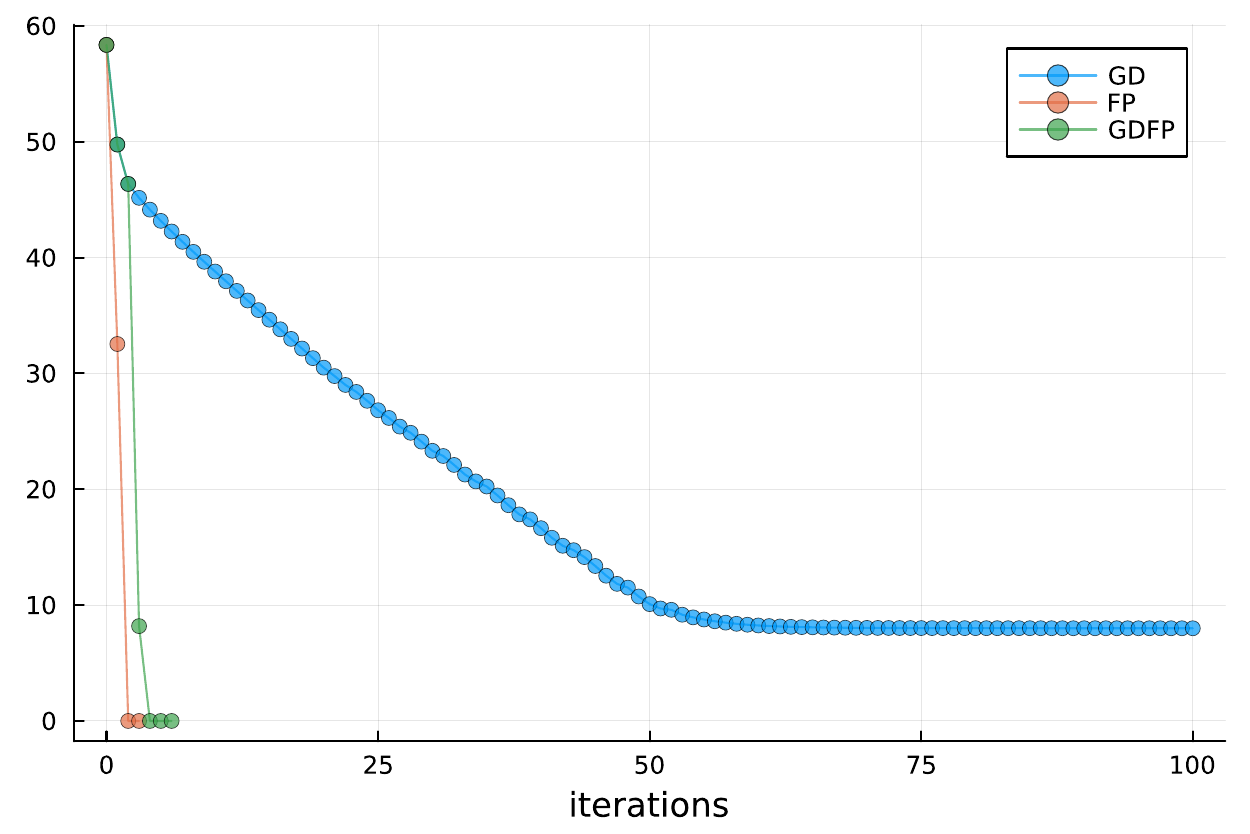}
        \caption{1x1}
        \label{fig:losses_1x1}
    \end{subfigure}
    \hfill
    \begin{subfigure}[t!]{0.4\textwidth}
        \includegraphics[width=\textwidth]{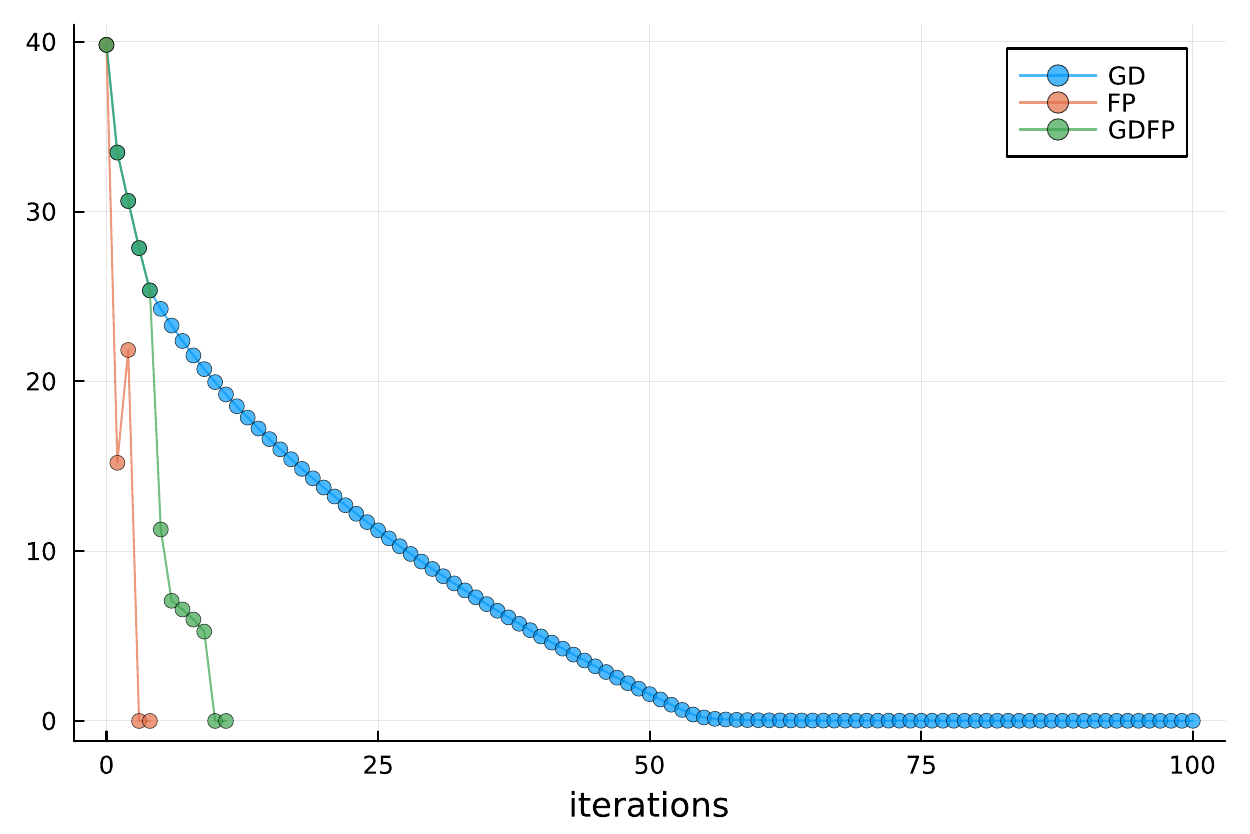}
        \caption{1x2}
        \label{fig:losses_1x2}
    \end{subfigure}
    \\
    \begin{subfigure}[t!]{0.4\textwidth}
        \includegraphics[width=\textwidth]{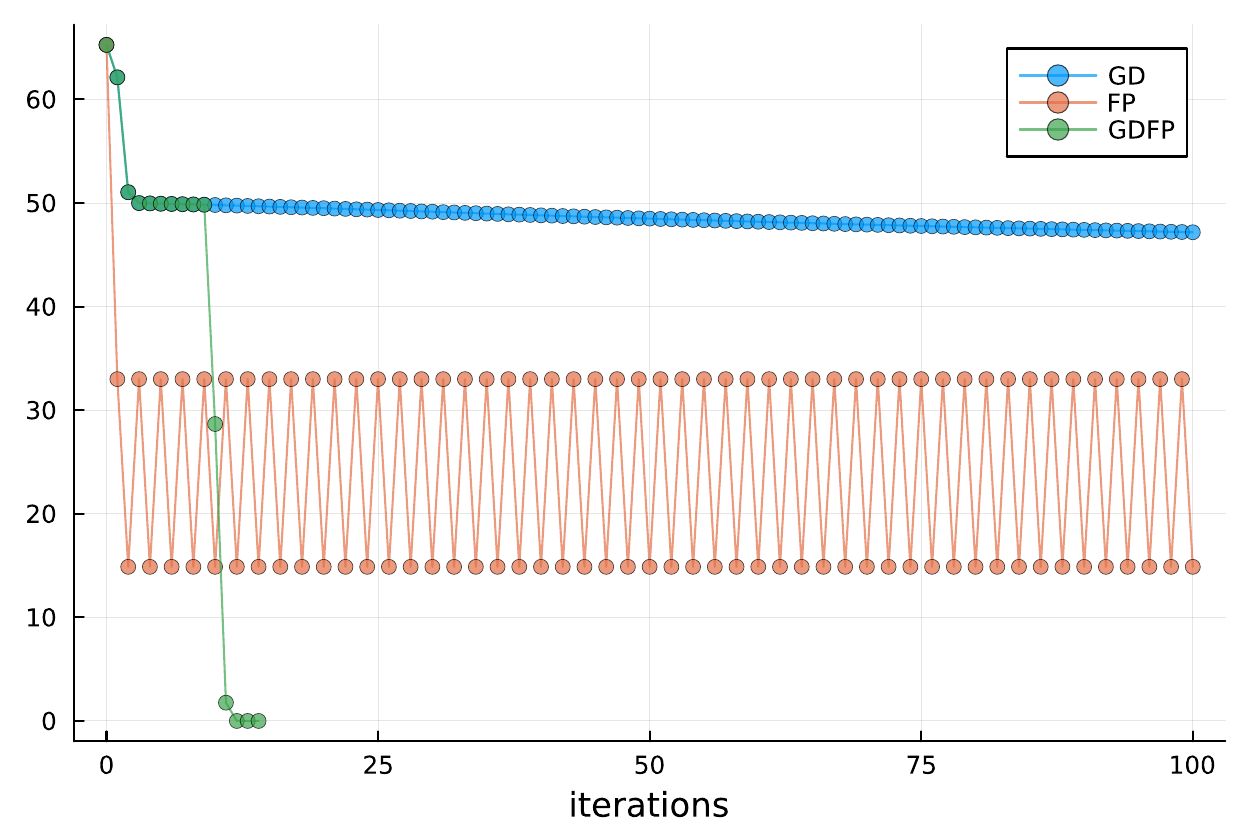}
        \caption{2x1}
        \label{fig:losses_2x1}
    \end{subfigure}
    \hfill
    \begin{subfigure}[t!]{0.4\textwidth}
        \includegraphics[width=\textwidth]{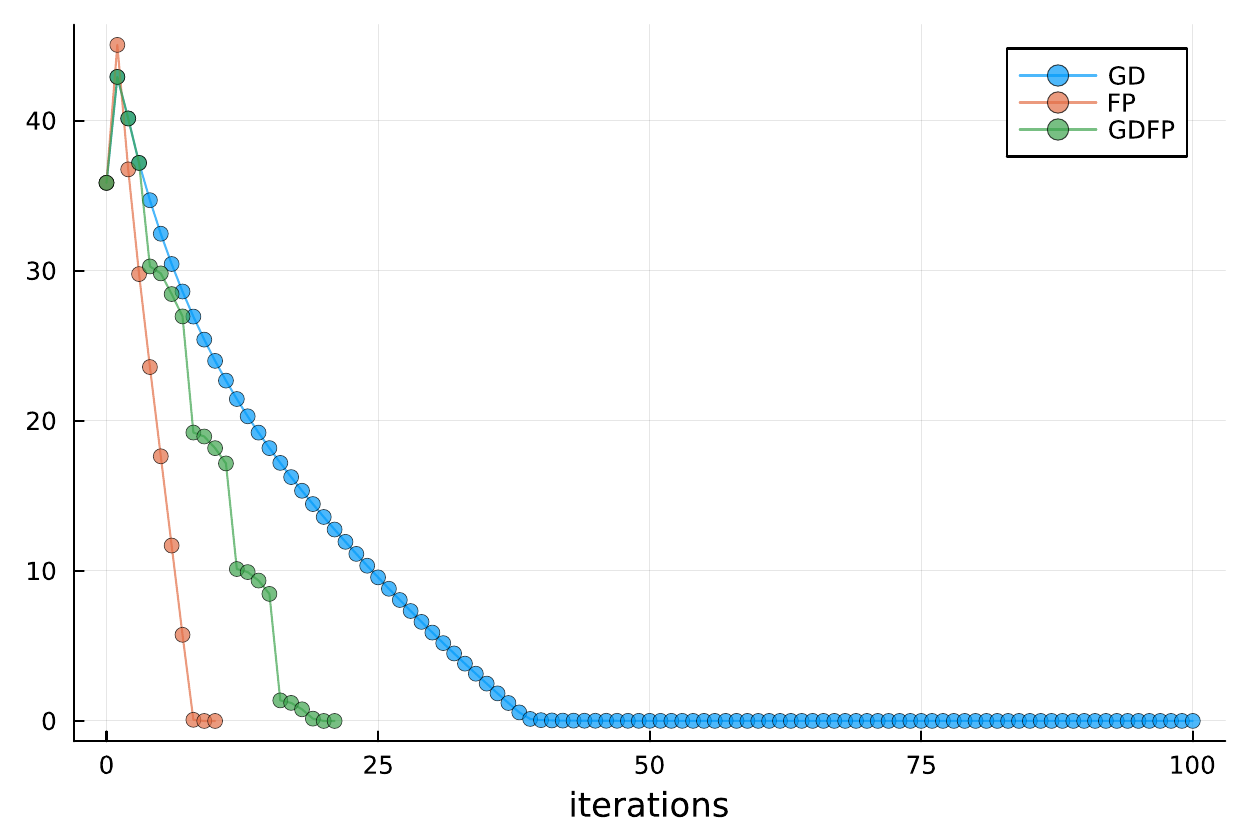}
        \caption{2x2}
        \label{fig:losses_2x2}
    \end{subfigure}
    \caption{(Single junctions) Cost functional as function of iterations for the gradient descent (GD), the fixed point (FP) and the hybrid (GDFP) methods. }
    \label{fig:single_junctions}
\end{figure}

\begin{figure}[H]
    \centering 
    \begin{subfigure}[t!]{0.4\textwidth}
        \includegraphics[width=\textwidth]{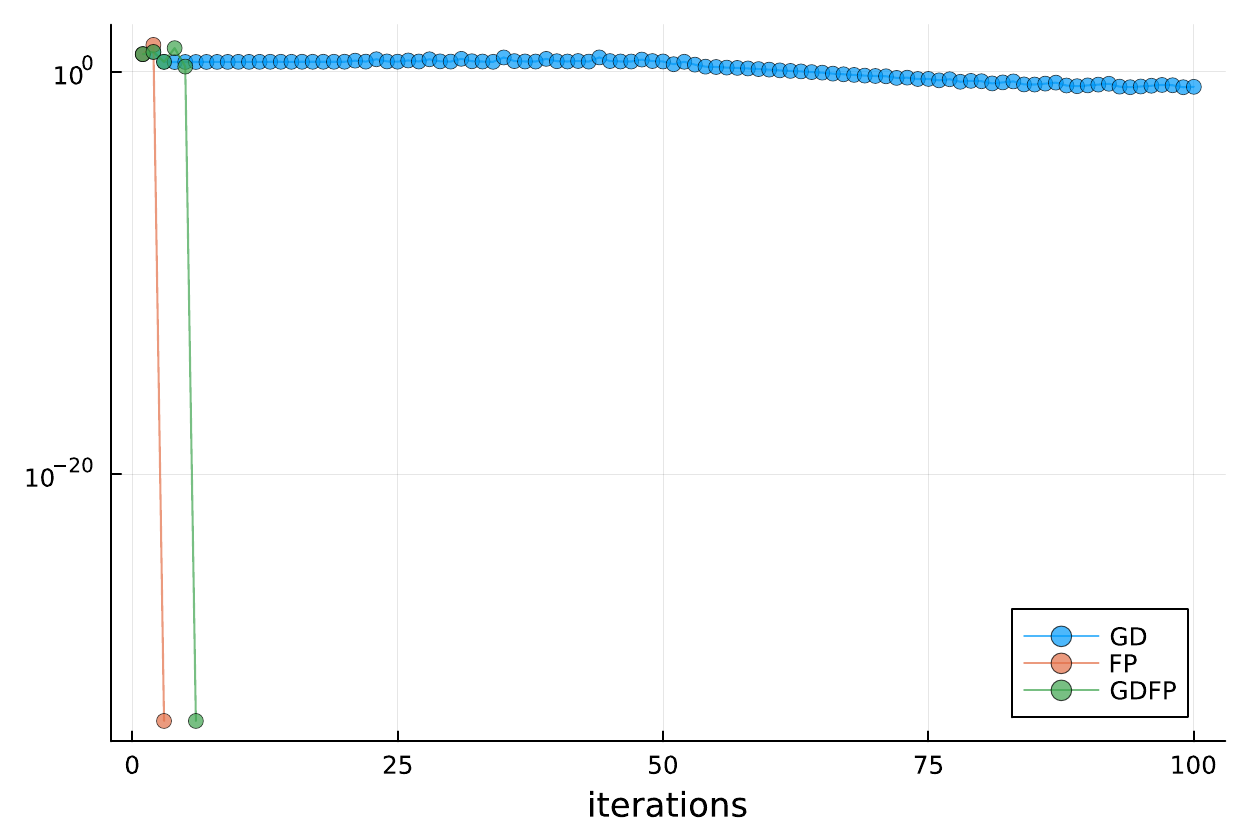}
        \caption{1x1}
        \label{fig:lambda_1x1}
    \end{subfigure}
    \hfill
    \begin{subfigure}[t!]{0.4\textwidth}
        \includegraphics[width=\textwidth]{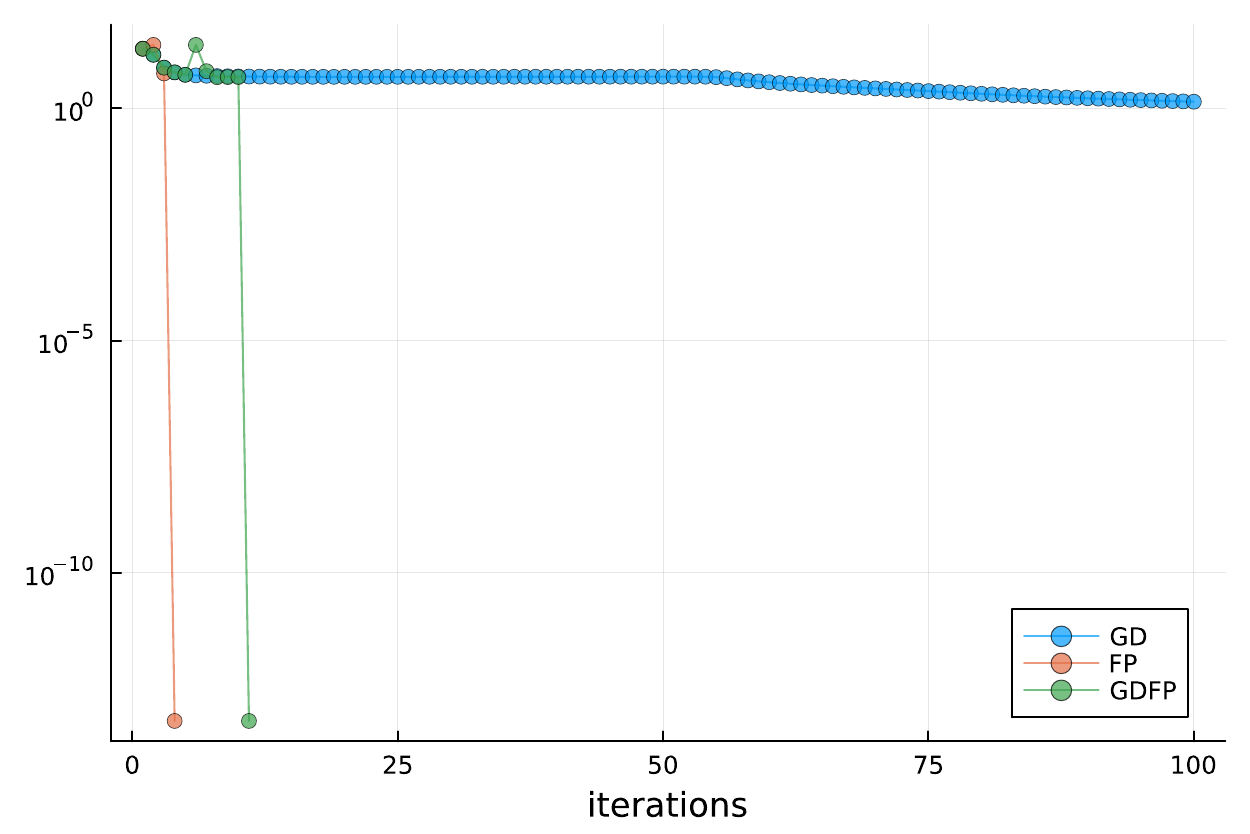}
        \caption{1x2}
        \label{fig:lambda_1x2}
    \end{subfigure}
    \\
    \begin{subfigure}[t!]{0.4\textwidth}
        \includegraphics[width=\textwidth]{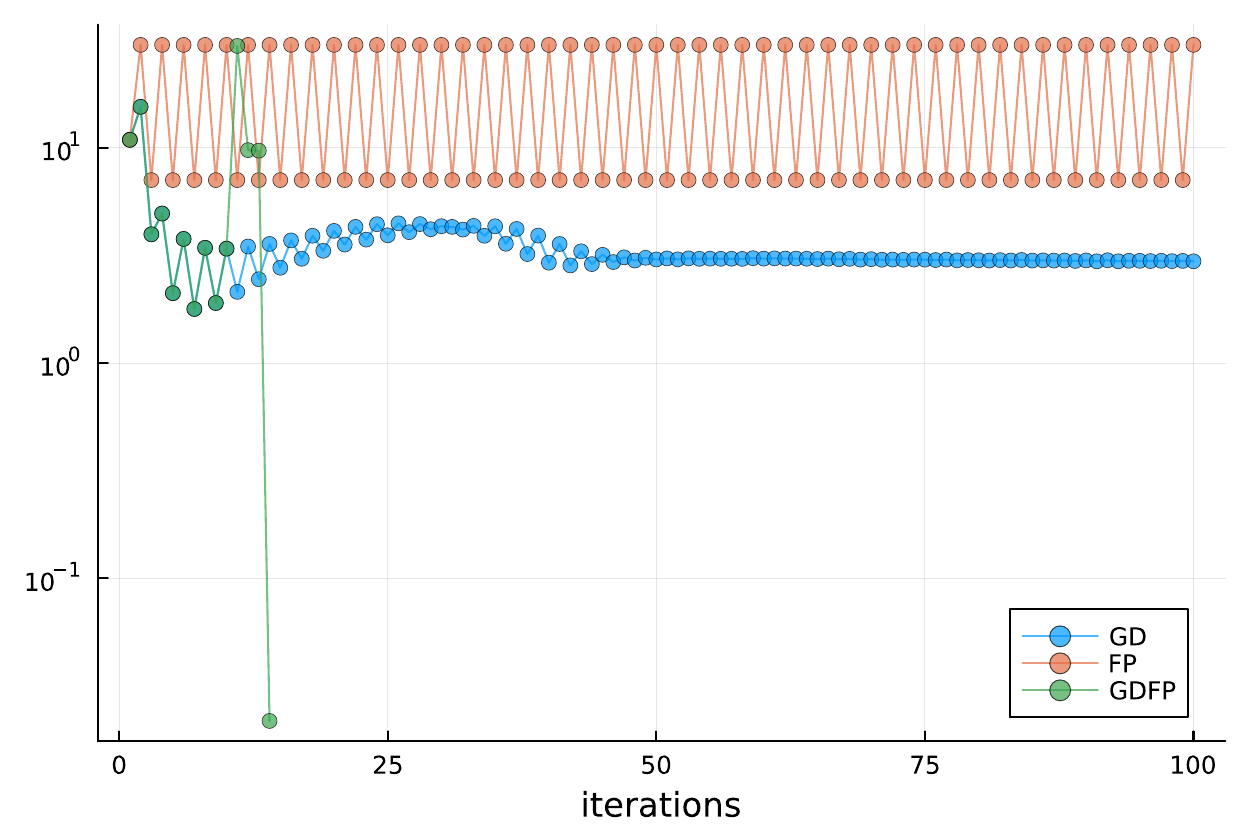}
        \caption{2x1}
        \label{fig:lambda_2x1}
    \end{subfigure}
    \hfill
    \begin{subfigure}[t!]{0.4\textwidth}
        \includegraphics[width=\textwidth]{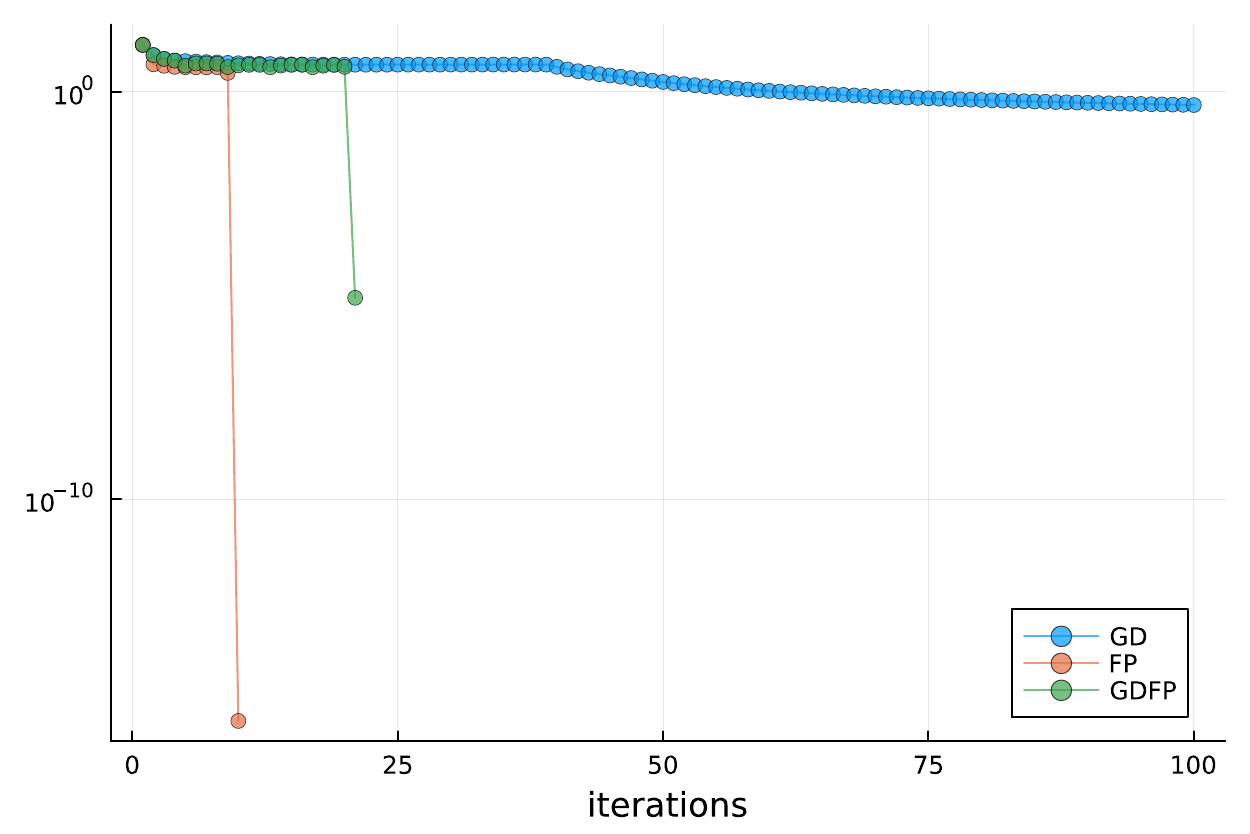}
        \caption{2x2}
        \label{fig:lambda_2x2}
    \end{subfigure}
    \caption{(Single junctions) Iterations of the convergence criterion $||\L||$.}
    \label{fig:single_junctions_lambda}
\end{figure}

\begin{figure}[H]
    \centering
    \begin{subfigure}[t!]{0.4\textwidth}
        \includegraphics[width=\textwidth]{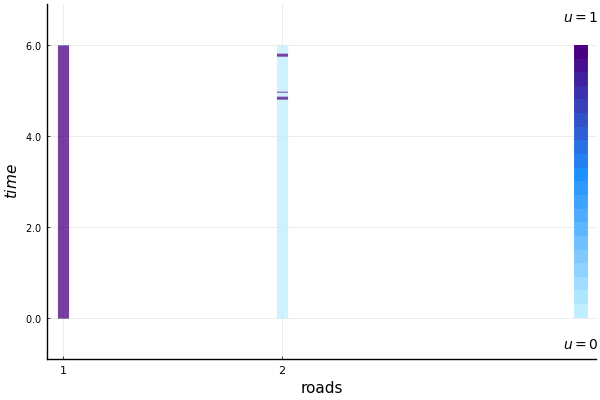}
        \caption{1x1}
        \label{fig:ctrl_1x1}
    \end{subfigure}
    \hfill
    \begin{subfigure}[t!]{0.4\textwidth}
        \includegraphics[width=\textwidth]{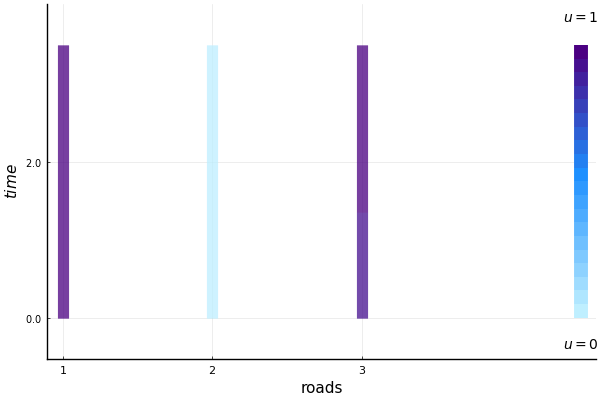}
        \caption{1x2}
        \label{fig:ctrl_1x2}
    \end{subfigure}
    \\
    \begin{subfigure}[t!]{0.4\textwidth}
        \includegraphics[width=\textwidth]{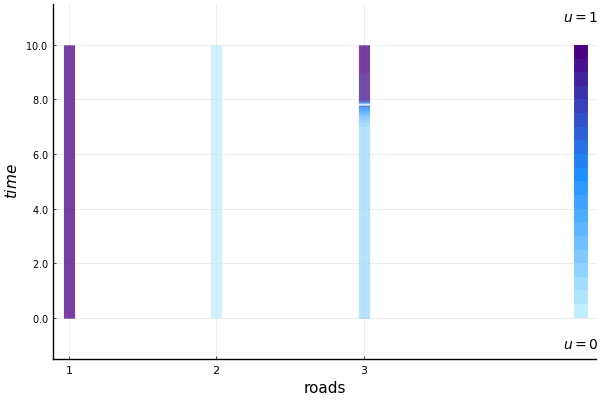}
        \caption{2x1}
        \label{fig:ctrl_2x1}
    \end{subfigure}
    \hfill
    \begin{subfigure}[t!]{0.4\textwidth}
        \includegraphics[width=\textwidth]{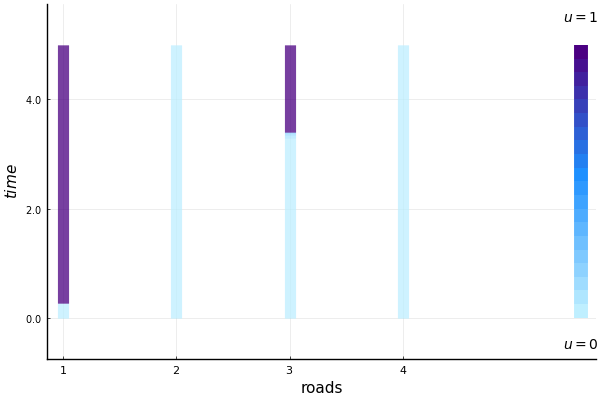}
        \caption{2x2}
        \label{fig:ctrl_2x2}
    \end{subfigure}
    \caption{(Single junctions) Optimal controls obtained with the GDFP method. The routes to empty are respectively $(1,2)$, $(1,3)$, $(1,3)$, $(1,3)$.}
    \label{fig:single_junctions_ctrl}
\end{figure}

\subsection{Traffic circle}

We consider the traffic circle test case as proposed in \cite{}. It is composed of $8$ roads as depicted in Fig.~\ref{fig:traffic_circle_route} including $2$ incoming roads, $2$ outgoing roads and the $4$ circle roads. Initial density is taken constant equal to $0.66$ on each road and, on Fig.~\ref{fig:traffic_circle_no_control}, we observe that congestion appears if no control is applied. 

The route to evacuate is chosen to be $(1,2,3,5,8)$ (see~Fig.~\ref{fig:traffic_circle_route}) and there is no constraints on the maximal number of controls ($\theta_S = 0$) or BV regularization ($\theta_B=0$). 
The parameters are given in \eqref{tab:threeways}. 

Results are gathered on Fig.~\ref{fig:traffic_circle_GDFP}. On the upper left panel, the cost functional history is depicted. The first four iterations of the gradient descent make the cost functional decrease very slowly, except for the second iteration, and then a FP step triggers the jump seen at iteration 5. A few GD steps are then observed and then another FP step results in a second jump at iteration 10. Then the gradient descent is able to reach a satisfactory local minimum by the 12th iteration. We note that the optimality function $\Lambda$ (Fig.~\ref{fig:traffic_circle_GDFP}) follows essentially the same behaviour and reaches a small value at the final iteration.
 
 The control obtained by the algorithm is given in  Fig.~\ref{fig:traffic_circle_GDFP_control}, bottom left. We observe first that roads $1$ and $6$ are allways controlled since they are entry roads on the network and would add new vehicles, and then that road $3$ entrance is always controlled to drive the flow to the outgoing road $4$. Furthermore, while road $5$ entrance is mostly controlled after time $t=8$ to let cars from road $3$ leave the route before, control on road $7$ entrance is quite the opposite: up to time $t=8$, it is activated to let the flow circulate as road $8$ entrance is open, then from time $t=8$ to $t=10$, it is deactivated as outgoing road $8$ entrance is now closed. Finally, on Fig.~\ref{fig:traffic_circle_GDFP_final_state}, we can check that the route is effectively empty at final time. 

\begin{table}[H]
    \centering
    \begin{tabular}{c|c|c|c|c|c|c|c|c|c|c}
    \toprule
    Parameters & $\vr_0$ & $\vu_0$ & $N_c$ & $\kappa$ & tol & $K_0$ & $\tau$ & T & $\delta_0$ & $\operatorname{decay}$  \\ 
    \midrule
    Traffic circle & $0.66$ & $1$ & $20$ & $10^{-6}$ & $10^{-2}$ & $5$ & 2 & $10$ & $5\times 10^{-2}$ & $10^{-1}$ \\
    Three lanes network & $0.66$ & $1$ & $5$ & $10^{-3}$ & $10^{-2}$ & $5$ & 3 & $30$ & $1$ & $10^{-2}$ \\ 
    \bottomrule
    \end{tabular}
    \caption{Numerical parameters for the \textit{traffic circle} and the \textit{three lanes} networks.}
    \label{tab:threeways}
\end{table}


\begin{figure}
    \centering
    \begin{subfigure}[t!]{0.4\textwidth}
        \includegraphics[width=\textwidth]{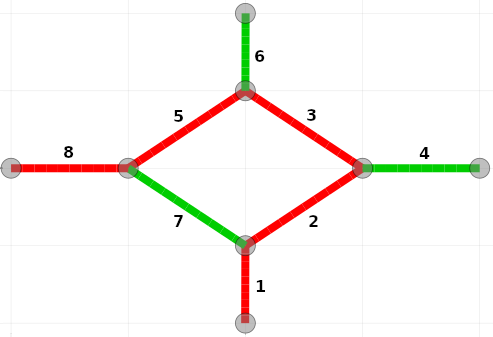}
        \caption{Road network and route to empty in red.}
         \label{fig:traffic_circle_route}
    \end{subfigure}
    \hfill
    \begin{subfigure}[t!]{0.4\textwidth}
        \includegraphics[width=\textwidth]{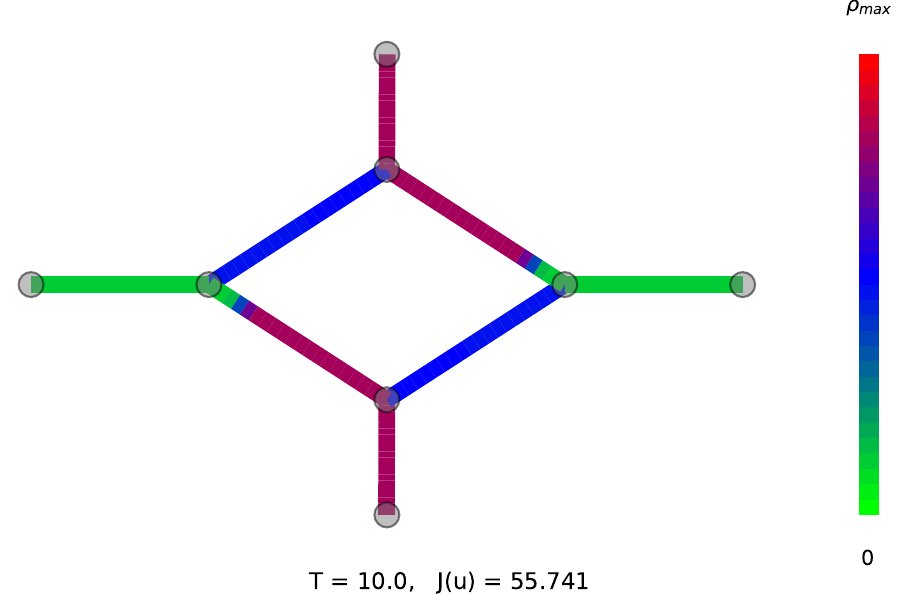}
        \caption{Density at final time without control.}
        \label{fig:traffic_circle_no_control}
    \end{subfigure}
    \caption{(Traffic circle) Left: Traffic circle configuration. The roads $(2,3,5,7)$ in the circle are counter-clockwise, roads $(1,6)$ are incoming and roads $(4,8)$ are outgoing. The route to empty $(1,2,3,5,8)$ is in red. Right: Numerical simulations without control at time $T=10$. }
    \label{fig:traffic_circle}
\end{figure}

\begin{figure}
    \centering
    \begin{subfigure}[t!]{0.4\textwidth}
        \includegraphics[width=\textwidth]{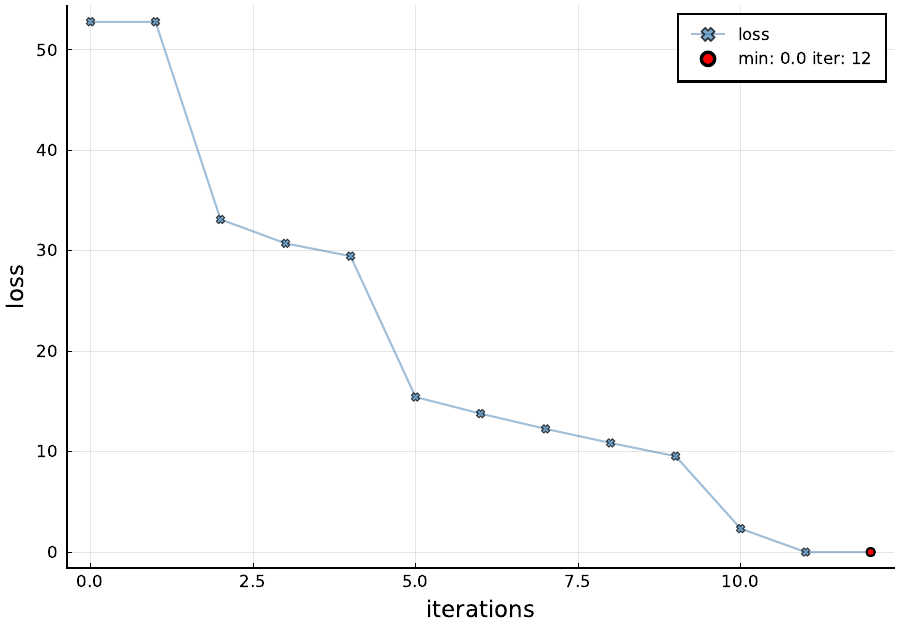}
        \caption{Cost functional history}
        \label{fig:traffic_circle_GDFP_cost}
    \end{subfigure}
    \hfill 
    \begin{subfigure}[t!]{0.4\textwidth}
        \includegraphics[width=\textwidth]{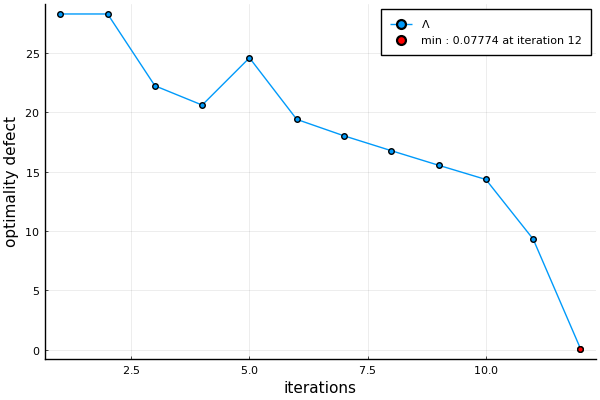}
        \caption{$||\L||_2$ history}
        \label{fig:traffic_circle_opt}
    \end{subfigure}\\
    \begin{subfigure}[t!]{0.4\textwidth}
        \includegraphics[width=\textwidth]{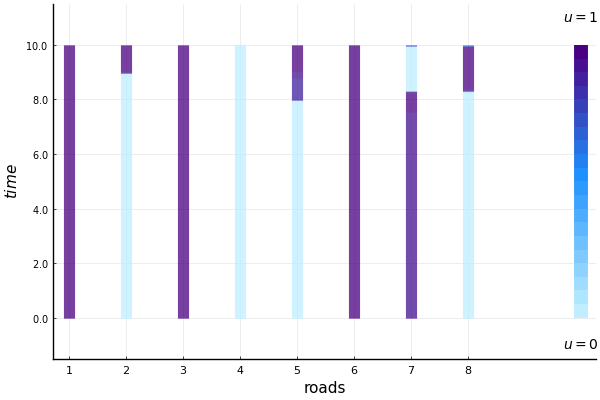}
        \caption{Control}
        \label{fig:traffic_circle_GDFP_control}
    \end{subfigure}
    \hfill 
    \begin{subfigure}[t!]{0.4\textwidth}
        \includegraphics[width=\textwidth]{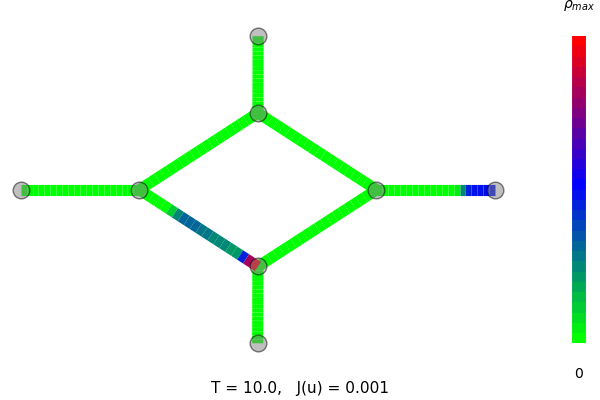}
        \caption{Density at final time}
        \label{fig:traffic_circle_GDFP_final_state}
    \end{subfigure}
    \caption{(Traffic circle). Results obtained with GDFP algorithm with spaced FP steps. }
    \label{fig:traffic_circle_GDFP}
\end{figure}

\subsection{Three lanes network}\label{sec:threeways}

We consider the three way network composed of $23$ roads as shown in Fig.~\ref{fig:threeways_redlane}: all the roads are directed from left to right. Initial density is sill constant equal to $0.66$ and a corresponding ingoing flux is imposed at the entrance of the network on road $1$. Neumann boundary conditions are used at the end of the outgoing road $11$. The statistical behaviors at each junction are uniform: there are no preferred trajectories.  Fig.~\ref{fig:threeways_no_ctrl} depicts the result of the simulation without control: the density on the central lane tends to be saturated as several roads lead onto it.

In the sequel, we would like to determine an optimal control in order to empty the central lane, made of roads $(4,6,7,8,9)$ (see Fig.~\ref{fig:threeways_redlane}), with $\Nm=5$.  Note that, contrary to the previous test cases, controlling this route may not prevent the flow to circulate from the ingoing to the outgoing road.  

\begin{figure}
    \centering
    \begin{subfigure}[t!]{0.4\textwidth}
        \includegraphics[width=\textwidth]{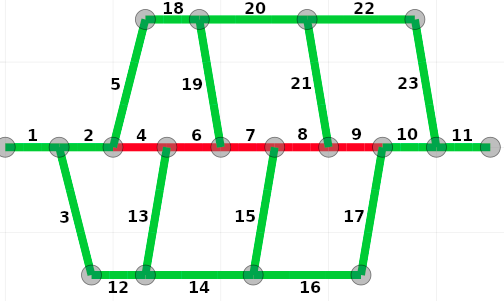}
        \caption{Road network and route to empty in red.}
        \label{fig:threeways_redlane}
    \end{subfigure}
    \hfill 
    \begin{subfigure}[t!]{0.4\textwidth}
        \includegraphics[width=\textwidth]{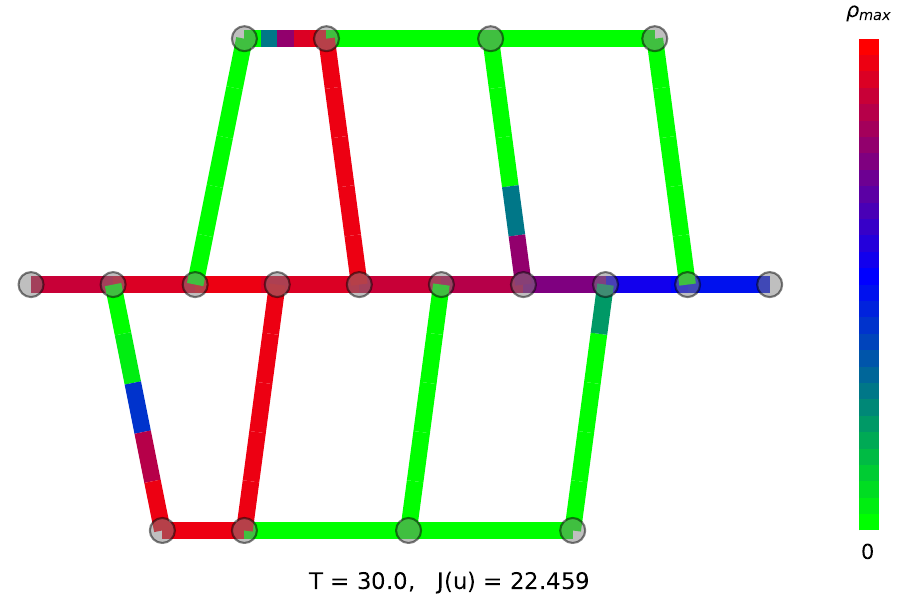}
        \caption{Density at final time without control.}
        \label{fig:threeways_no_ctrl}
    \end{subfigure}
    \caption{(Three lanes network) Left: Three lanes network configuration. Road $1$ is incoming and road $11$ is outgoing and all the others are directed from left to right. The route to empty $(4,6,7,8,9)$ is in red. Right: Numerical simulations without control at time $T=30$.}
    \label{fig:threeways}
\end{figure}

First we consider the optimal problem with essentially no limitation on the number of active controls ($\theta_S = 10^{-8}$) nor BV regularization ($\theta_B = 10^{-8}$). When $\Vert \Lambda\Vert_2$ becomes lower than $1$ we consider that we are close to a basin of attraction, and we give increased effect to regularization by setting $\theta_S=10^{-4}$ and $\theta_B=10^{-6}$. We use the numerical parameters of \eqref{tab:threeways}: in particular the FP steps are used at iterations $5$ and $15$. 

Fig.~\ref{fig:threeways_GDFP_loss} represents the cost functional history during the optimization process. We observe first a rapid decrease of the cost functional during the gradient descent steps, which reaches a basin of attraction at iteration $9$ with $\J=0.23$ and $\Vert \Lambda\Vert_2 = 1.001\times 10^{-2}$. Fig.~\ref{fig:threeways_GDFP_control_best} shows the control provided by the algorithm for this iteration. The time evolution of the road densities with this control is represented in Fig.~\ref{fig:threeways_GDFP_mean_best}. The control performs rather well regarding the final cost functional and the final road densities on the central line. However, as might be expected, this remains rather unsatisfactory given the relatively high number of active controls and the redundancies. For instance, it should be useless to block the entrance of roads 18, 20 and 22 if roads 19, 21 and 23 are controlled.

We then observe a slight increasing of the loss $\J$ when the regularization coefficients $\theta_S$ and $\theta_B$ are increased, mainly due to the high values of $S(\vu)$ and $B_\nu(\vu)$. After the fixed-point step at iteration $15$, all the coefficients of the loss are decreasing and reach the convergence tolerance by the $17$-th iteration, as shown in Fig.~\ref{fig:threeways_GDFP_optimality}. A zoom on this phenomena is given Fig.~\ref{fig:zoomGDFP}.

The cost $C_T(\vu)$ remains equal to $0.23$ whereas the FP step at iteration $15$ clearly reduced the staffing constraint $S(\vu)$ by more than an order of magnitude, at the cost of a slight increase in the total variation. Note that the GD steps were stationnary and that only the FP step led to convergence.

\begin{figure}[H]
    \centering
    \begin{subfigure}[t!]{0.4\textwidth}
        \includegraphics[width=\textwidth]{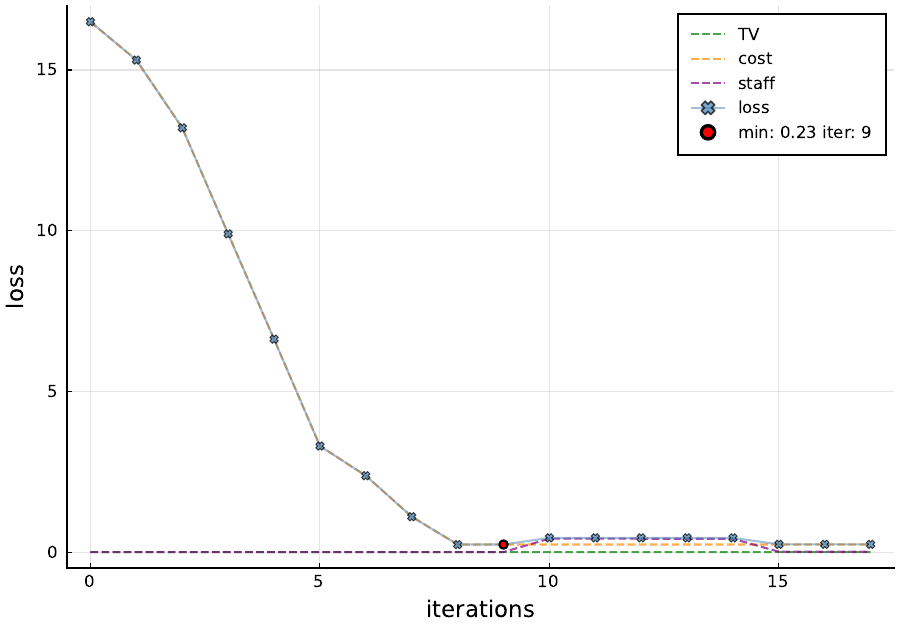}
        \caption{Cost functional.}
        \label{fig:threeways_GDFP_loss}
    \end{subfigure}
    \hfill 
    \begin{subfigure}[t!]{0.4\textwidth}
        \includegraphics[width=\textwidth]{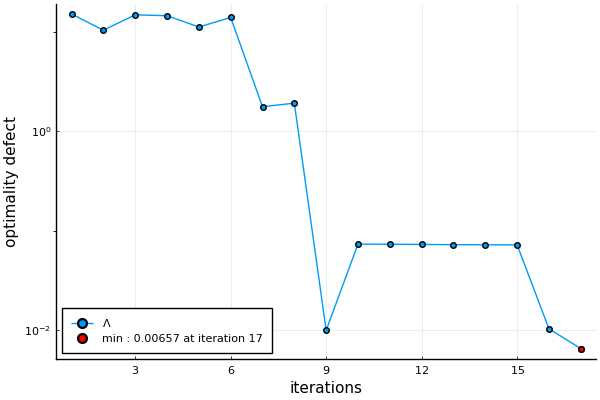}
        \caption{$||\L||_2$ history.}
        \label{fig:threeways_GDFP_optimality}
    \end{subfigure}
    \\
    \centering
       \begin{subfigure}[t!]{0.4\textwidth}
        \includegraphics[width=\textwidth]{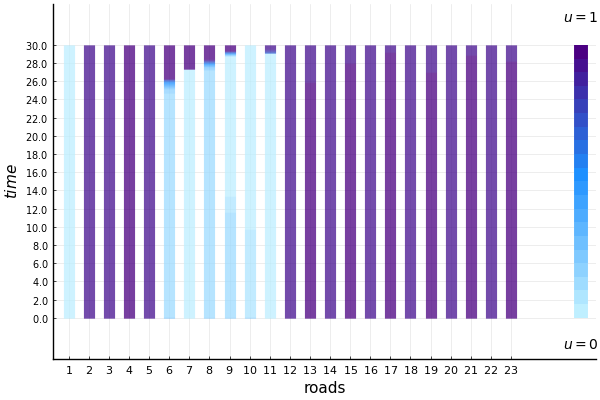}
        \caption{Control (low constraints).}
        \label{fig:threeways_GDFP_control_best}
    \end{subfigure}
    \hfill
    \begin{subfigure}[t!]{0.4\textwidth}
        \includegraphics[width=\textwidth]{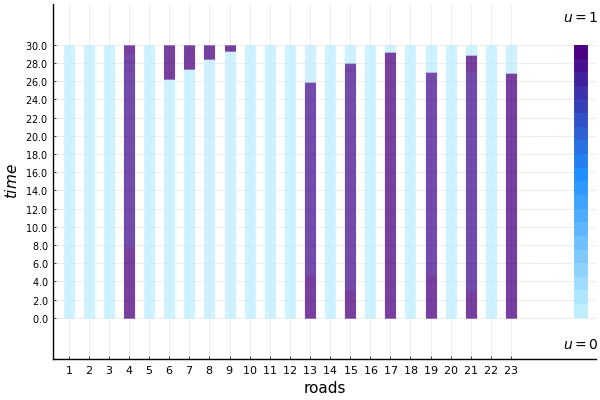}
        \caption{Control (increased constraints).}
    \end{subfigure}
    \caption{(Three lanes network with scheduled constraints) Controls obtained with the GDFP algorithm.} 
    \label{fig:threeways_GDFP_1}
\end{figure}

\begin{figure}[H]
    \centering 
    \begin{subfigure}[t!]{0.4\textwidth}
        \includegraphics[width=\textwidth]{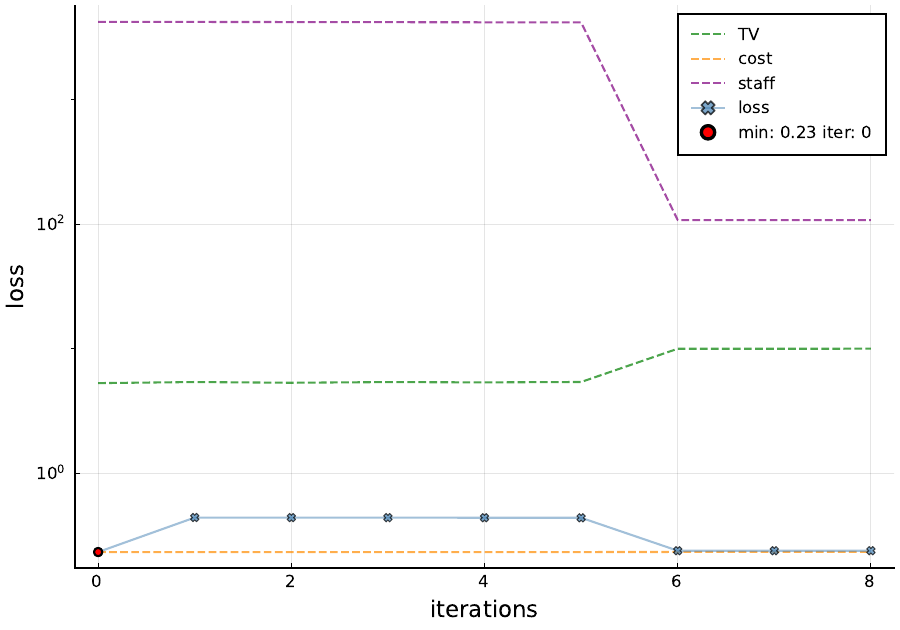}
        \caption{Without weights.}
        \label{fig:threeways_GDFP_loss_zoom}
    \end{subfigure}
    \hfill 
    \begin{subfigure}[t!]{0.4\textwidth}
        \includegraphics[width=\textwidth]{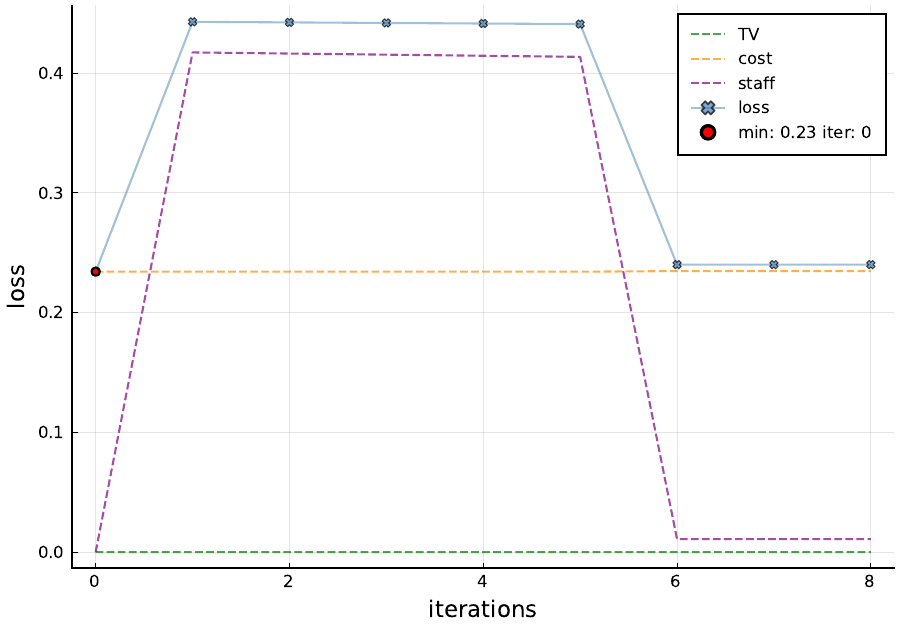}
        \caption{With weights.}
        \label{fig:threeways_GDFP_loss_zoom_weighted}
    \end{subfigure}
    \caption{Cost functional between iterations $9$ and $17$.}
    \label{fig:zoomGDFP}
\end{figure}

\begin{figure}[H]
    \centering
    \begin{subfigure}[t!]{0.4\textwidth}
        \includegraphics[width=\textwidth]{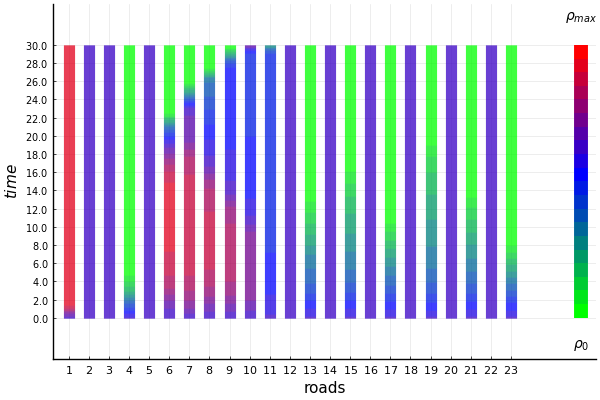}
        \caption{Mean density (low constraints).}
        \label{fig:threeways_GDFP_mean_best}
    \end{subfigure}
    \hfill 
    \begin{subfigure}[t!]{0.4\textwidth}
        \includegraphics[width=\textwidth]{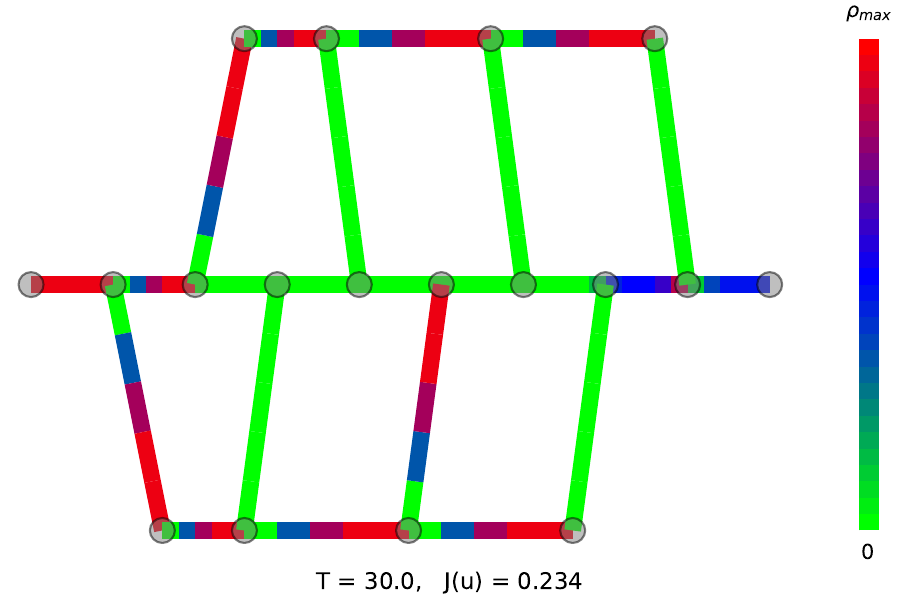}
        \caption{Final time snapshot (low constraints).}
    \end{subfigure}
    \\
    \begin{subfigure}[t!]{0.4\textwidth}
        \includegraphics[width=\textwidth]{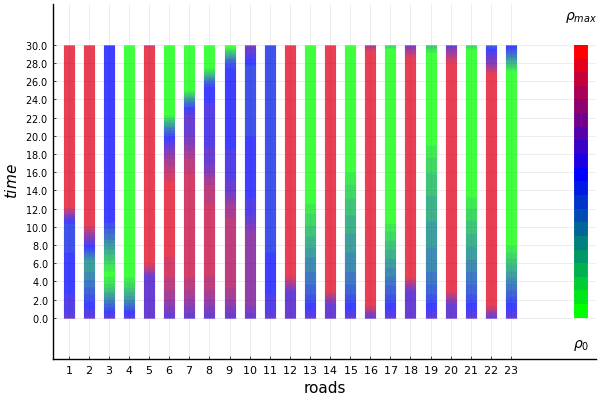}
        \caption{Mean density (increased constraints).}
    \end{subfigure}
    \hfill 
    \begin{subfigure}[t!]{0.4\textwidth}
        \includegraphics[width=\textwidth]{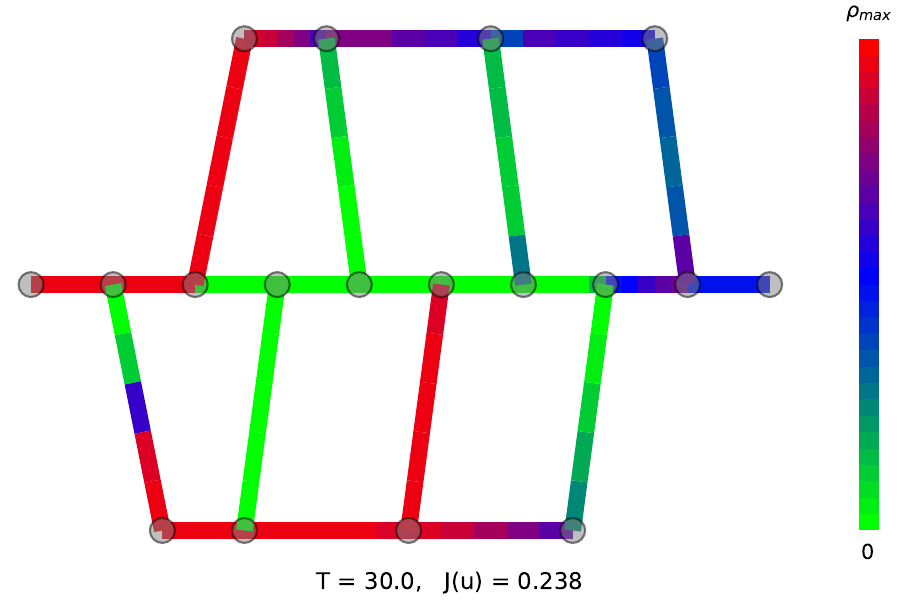}
        \caption{Final time snapshot (increased constraints).}
    \end{subfigure}
    \caption{(Three lanes network with scheduled constraints) Densities obtained with the GDFP algorithm.} 
    \label{fig:threeways_GDFP_2}
\end{figure}

The penalization algorithm has clearly solved the two issues: the number of active controls is now less than $7$, which is still higher than $5$ but far preferable than the earlier $16$ in a approached control point of view, and there is no time oscillations of the controls even if it would have been a mathematically correct way to reduce the staffing constraint. 

\begin{table}[H]
    \centering
        \begin{tabular}{c|l|c}
            \toprule
            Symbol & Name & Value \\ 
            \midrule
            $\Nm$ & maximum number of simultaneous blockages allowed & 5 \\
            $\theta_S^0$ & initial coefficient for the staffing constraint  & $10^{-8}$\\
            $\theta_S$ & scheduled coefficient for the staffing constraint & $10^{-4}$\\
            $\theta_B^0$ & initial coefficient for the BV constraint & $10^{-8}$ \\
            $\theta_B$ & scheduled coefficient for the BV constraint & $10^{-6}$ \\
            $\nu$ & smoothed absolute value coefficient & $10^{-10}$ \\
            \bottomrule
        \end{tabular}
    \caption{Numerical parameters for the full model.}
	\label{tab:extension} 
\end{table}

\section{Conclusion}
\label{sec:conclusion}

This article has therefore introduced a new network traffic control problem where the control acts not only on the flux constraints at the entrance of the road but also on the statistical behaviour matrices at the junctions. The theoretical analysis of the optimal control problem led to a hybrid numerical method combining the gradient descent method and a fixed-point strategy. The numerical results obtained on the  traffic circle and on the three-lane network have shown that the method is relatively effective in providing relevant controls. It is also possible to impose constraints on the number of active controls.

This work will be pursued in several directions. Firstly, it would be interesting to show theoretically that optimal controls are indeed bang-bang, i.e. that they take only the two extreme values 0 and 1, as suggested by our numerical results. This could lead to further improvements in the optimization algorithm. For example, we plan to formulate a shape optimization problem whose unknown is the (temporal) domain over which control is equal to 1, and then apply efficient algorithms for such problems, such as the level set method.   

Secondly, the indicator and projection functions make our algorithms sensitive to threshold or scaling effects. Further analysis could prove very useful in generically determining hyperparameters such as $\theta$ or $\kappa$, given that they are fixed by hand for the time being. Finally, the problem has been tackled here using a strategy of discretization (in space) followed by optimization. Another possibility would have been to optimize first and then discretize. Such a strategy may give rise to a different type of algorithm. 

A longer-term objective we plan to tackle in the coming years is to generalize and test our algorithm both on larger road networks and on real data. To this end, we will seek to couple this method with relevant initialization methods for which initial control would be obtained by considering a macroscopic graph and implementing model reduction methods. Such an approach has, for example, been developed in a different context in epidemiology~\cite{courtes:hal-03664271}.

\appendix
\begin{center}
\fbox{\textsf{\Large Appendix}}
\end{center}

\section{Expression of the function $\phi^{LP}$}\label{append:modelCont}
We provide hereafter the explicit expression of $\phi^{LP}$ for junctions with at most 2 ingoing and 2 outgoing roads. We will write $ (\vg^R, \vg^L)= \phi^{LP}(\vr,\vu)$.

\paragraph{One ingoing and one outgoing roads ($n=m=1$).}
In that case, the solution of the LP problem \eqref{pb:mainLP}
 is given by
\begin{equation*}\label{eq:one-one}
    \g_{1}^R = \g_2^L = \min\left(\g_{1}^{R,\max}, (1-u_2)\g_{2}^{L,\max}\right).
\end{equation*}
 
\paragraph{One ingoing and two outgoing roads ($n=1$, $m=2$).}
Setting $\a=\a_{21}$, we obtain
\begin{align*}
 \g_1^R &= \min\left( \g_1^{R,\max}, \min\left( (1-u_2)\frac{\g_2^{L,\max}}{\a}, (1-u_3)\frac{\g_3^{L,\max}}{1-\a} \right) \right), \\ 
 \g_2^L &= \a \g_1^R, \ \g_3^L = (1-\a) \g_1^R.
\end{align*}
However, this formulation raises a modeling problem. Indeed, this solution does not distinguish between a blocked road ($\g_j^{R,\max}=0$) and a controlled road ($1-u_j=0$). Indeed, the first case corresponds to a congestion involving cars on both sides of the intersection, preventing vehicles from entering the second -potentially empty- road, while in the second case, drivers will not try to enter the controlled road and we revert to the $1\times 1$ configuration with the remaining roads. 

To capture this difference in driver behavior, we will adjust the parameter $\a$ with respect to the control $\vu$. Therefore, $\alpha$ will be considered as a function of $(u_2,u_3)$. More precisely, 
we introduce a function denoted $P$ of the variable $x= u_2-u_3$, taking the constant value $\bar \a\in (0,1)$ at $x=0$ and yielding the relevant $1\times 1$ case whenever one road is fully controlled. For the sake of simplicity, we look for a function $(u_2,u_3) \mapsto P(u_2 - u_3)$, where $P$ is a polynomial of degree at most 2, that satisfies
$$
P(-1) = \a(0,1) = 1, \quad P(0)  = \a(u,u) =\bar \a,\quad P(1)  = \a(1,0) = 0.
$$
The standard Lagrange interpolation formula
yields
\begin{equation}
    P(x) = \frac{x(x-1)}{2} +\bar \a(1-x^2).
\end{equation}
However, due to the division by $\alpha$ in the expression of $\gamma_1^R$, the case $\a(u_2=1, u_3=0) = 1-u_2 = 0$ is degenerate. To overcome this definition problem, we introduce a small perturbation parameter $\e>0$ and a modification of the polynomial $P$ denoted $P_\e$, such that $\a_\e := P_\e(u_2-u_3)$ and
\begin{equation*}
    P_\e(-1) = 1-\e^2, \quad P_\e(0) = \a, \quad P_\e(1) = \e^2.
\end{equation*}
The same reasoning as above yields
$
    P_\e(x) = P(x) + \e^2 x.
$
Since the interpolation procedure does not guarantee that the range of the function is contained in $[\e^2, 1-\e^2]$, we compose the obtained expression with a projection onto the set of admissible values, to get at the end 
\begin{equation}
    \a_\e(u_2,u_3) := \operatorname{proj}_{[\e^2, 1-\e^2]}\left(P_\e(u_2-u_3)\right).
\end{equation}
Similarly, we also replace the expression $(1-u_j)$ by $(1-u_j + \e)$ to avoid the case where $(1-u_2)=\a(u_2,u_3)=0$. 
We finally obtain the following regularized solution at $1\times 2$ junctions
\begin{align*}
	\g_1^R &= \min\left( \g_1^{R,\max}, \min\left( (1-u_2+\e)\frac{\g_2^{L,\max}}{\a_\e(u_2,u_3)}, (1-u_3+\e)\frac{\g_3^{L,\max}}{1-\a_\e(u_2,u_3)} \right) \right), \\ 
	\g_2^L &= \a_\e \g_1^R, \ \g_3^L = (1-\a_\e) \g_1^R.
\end{align*}

\paragraph{Two ingoing and one outgoing roads ($n=2$, $m=1$).} 
The ingoing roads are indexed by $1,2$ and the outgoing one is indexed by 3. In any case, we have $\g_3^L = \g_1^R + \g_2^R$, and $\g_1^R, \g_2^R$ are computed as follows:
\begin{itemize}
    \item we set $\g_{3,\vu}^{L,\max} := (1-u_3)\,\g_3^{L,\max}$;
    \item if $\g_1^{R,\max} + \g_2^{R,\max} \leq \g_{3,u}^{L,\max}$, then $\g_1^R = \g_1^{R,\max}$ and $\g_2^R = \g_2^{R,\max}$;
    \item else,
        \begin{itemize}
            \item if $\g_1^{R,\max} \geq q\,\g_3^{L,\max}$ and $\g_2^{R,\max} \geq (1-q)\g_{3,u}^{L,\max}$, then $\g_1^R = q \, \g_{3,u}^{L,\max}$ and  $\g_2^R = (1-q) \, \g_{3,u}^{L,\max}$.
        \item if $\g_1^{R,\max} < q\g_{3,u}^{L,\max}$ and $\g_2^{R,\max} \geq (1-q)\g_{3,u}^{L,\max}$, then $\g_1^R = \g_1^{R,\max}$ and $\g_2^R = \g_3^{L,\max} - \g_1^{R,\max}$.
        \item if $\g_1^{R,\max} \geq q\g_{3,u}^{L,\max}$ and $\g_2^{R,\max} < (1-q)\g_{3,u}^{L,\max}$, then $\g_1^R = \g_{3,u}^{L,\max} - \g_2^{R,\max}$ and $\g_2^R = \g_2^{R,\max}$.
        \end{itemize}
\end{itemize}

\paragraph{Two ingoing and two outgoing roads ($n=m=2$).} 
The ingoing roads are indexed by $1,2$ and the outgoing ones are indexed by $3,4$. We have \begin{equation*}
    \begin{pmatrix} \g_3 & \g_4 \end{pmatrix}^\top = A(\vu)\begin{pmatrix} \g_1 & \g_2 \end{pmatrix}^\top
    \quad \text{with}\quad     A(\vu)=\begin{pmatrix} \a(\vu) & \b(\vu) \\ 1-\a(\vu) & 1-\b(\vu) \end{pmatrix},
\end{equation*} 
where, given two real number $\bar \a$, $\bar \b$ in $(0,1)$, the coefficients are defined in the same way as in the case $1\times 2$, as
$$
    \a_\e(u) = \operatorname{proj}_{[\e^2, 1-\e^2]} \left(P^{\bar \a}_\e(u_3-u_4)\right),\quad
    \b_\e(u) = \operatorname{proj}_{[\e^2, 1-\e^2]} \left(P^{\bar \b}_\e(u_3-u_4)\right).
$$
Here, $P^\xi_\e$ is defined by
\begin{equation}\label{eq:poly}
    P^\xi_\e(x) = \frac{x(x-1)}{2} + \xi(1-x^2) + \e^2 x, \quad \xi\in\{\bar \a,\bar \b\}.
\end{equation}

Let us set $\g_{i,\vu}^{L,\max} := (1-u_i)\,\g_i^{L,\max}$, for $i=3,4$.

To solve the LP $2\times 2$ problem, we first analyze the simplex standing for the polytope of constraints, represented on \Cref{fig:lp2x2easy}: the constraints related on the incoming routes form the rectangle $\Om_{in} := [0,\g_1^{R,\max}]\times[0,\g_2^{R,\max}]$ while those related on the outgoing roads correspond to the regions below the curves $\mathcal{C}_3$ and $\mathcal{C}_4$, respective graphs of the functions
\[
 \g_1 \mapsto \frac{(1-u_3)\g_3^{L,\max}-\au\g_1}{\bu} \quad\text{and}\quad \g_1 \mapsto \frac{(1-u_4)\g_4^{L,\max}-(1-\au)\g_1}{1-\bu}.
\]
A convexity argument, standard in linear optimization, yields that the exists (at least) a solution lying on a vertex of the set of constraints.  To discuss on optimality of each vertex, we first compute $\vG_\vu$, the intersection of $\mathcal{C}_3$ and $\mathcal{C}_4$.

\begin{figure}[h]
    \centering
    \begin{tikzpicture}[line cap=round, line join=round, >=triangle 45, x=20.0cm, y=20.0cm]
    \draw[->,color=black] (-0.1,0) -- (0.4,0);
    \foreach \x in {,0.1,0.2,0.3}
    \draw[->,color=black] (0,-0.05) -- (0,0.3);
    \foreach \y in {,0.1,0.2,0.3}
    \draw[color=black] (0pt,-10pt) node[right] {\footnotesize $0$};
    \clip(-0.11,-0.05) rectangle (0.4,0.3); 
    \fill[color=zzttqq,fill=zzttqq,fill opacity=0.1] (0,0) -- (0.25,0) -- (0.25,0.25) -- (0,0.25) -- cycle;
    \draw [domain=-0.11:0.4] plot(\x,{(--0.25-0*\x)/1});
    \draw (0.25,-0.05) -- (0.25,0.3);
    \draw [color=qqwwff,domain=-0.11:0.4] plot(\x,{(--0.51-1.27273*\x)/1});  
    \draw [color=ffwwqq,domain=-0.11:0.4] plot(\x,{(--0.40071-0.78571*\x)/1}); 
    \draw [color=ffqqqq,domain=-0.11:0.4] plot(\x,{(--0.44989-1*\x)/1});
    \draw [color=zzttqq] (0,0)-- (0.25,0);
    \draw [color=zzttqq] (0.25,0)-- (0.25,0.25);
    \draw [color=zzttqq] (0.25,0.25)-- (0,0.25);
    \begin{scriptsize} 
        \draw[] (0.39,-0.02) node {$\g_1$}; 
      \draw[] (-0.02,0.28) node {$\g_2$};
        \draw[color=ffwwqq] (0.35,0.15) node {$\mathcal{C}_3$}; 
    \draw[color=qqwwff] (0.30,0.10) node {$\mathcal{C}_4$};
    \fill [color=zzqqzz] (0.2244,0.2244) circle (1.5pt);
    \draw[color=zzqqzz] (0.24,0.235) node {$\vG_\vu$};
    \end{scriptsize}
    \end{tikzpicture}
    \caption{Polytope of constraints for the 2x2 LP.
    The constraints on the input flow are represented by the colored area.\label{fig:lp2x2easy}}   
\end{figure}
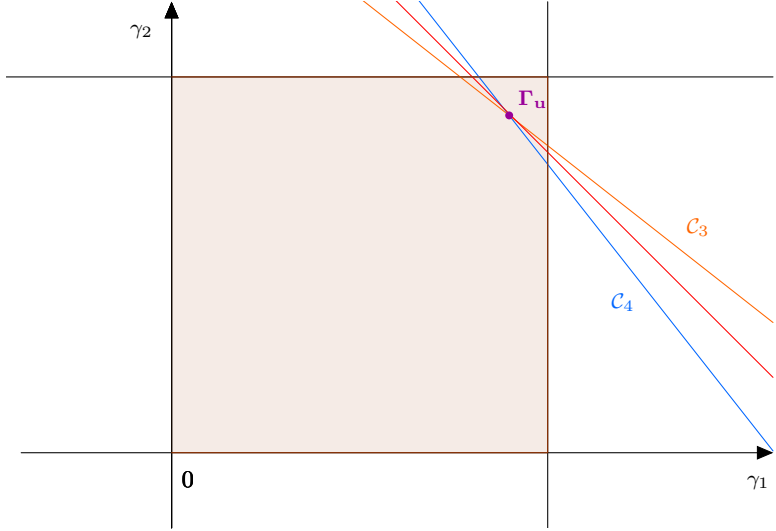

The point $\vG_{\vu} = (\G_{1\vu}, \G_{2\vu})$ is the intersection point of the lines with cartesian equations
$\au\g_1+\bu\g_2 = \g_{3,\vu}^{\max}$ and $(1-\au)\g_1 + (1-\bu)\g_2 = \g_{4,\vu}^{\max}$.
We obtains
\begin{equation}
    \begin{cases} 
        \G_{1\vu} = \left( (1-\bu)\g_{3,\vu}^{\max} - \bu\g_{4,\vu}^{\max} \right) / \D_{\vu}, \\ 
        \G_{2\vu} = \left( -(1-\au)\g_{3,\vu}^{\max} - \au\g_{4,\vu}^{\max} \right) / \D_{\vu}, \\  
    \end{cases} 
\end{equation} 
with $\D_{\vu} = \au(1-\bu) - \bu(1-\au)$.
A standard graphical reasoning shows that the solution is realized at $\vG_\vu$ if $\vG_\vu$ belongs to $\Om_{in}$, as in~Fig.~\ref{fig:lp2x2easy}. Otherwise, we have to intersect the constraint of highest slope (in absolute value) with $\Om_{in}$. It is this disjunction that is performed in what follows.

\begin{itemize}
    \item {\bf First case:} $\G_{1\vu}\leq\g_1^{\max}$ and $\G_{2\vu}\leq\g_2^{\max}$;
    \begin{itemize}
        \item if $\G_{1\vu} >= 0$ and $\G_{2\vu}<0$, then $\g_2=0$;
            \begin{itemize}
                \item if $\au < \bu$, then $\g_1=\g_{3,\vu}^{\max} / \au$;
                \item else, $\g_1=\g_{4,\vu}^{\max} / (1-\au)$.
            \end{itemize}
        \item if $\G_{1\vu} < 0$ \textbf{and} $\G_{2\vu} \geq 0$, then $\g_1=0$;
            \begin{itemize}
                \item if $\au < \bu$, then $\g_2 = \g_{3\vu}^{\max} / \bu$;
                \item else, $\g_2 = \g_{4\vu}^{\max} / (1-\bu)$.
            \end{itemize}
        \item if $\G_{1\vu}<0$ \textbf{and} $\G_{2\vu}<0$, then $\g_1 = \g_2 = 0$;
        \item else, $(\g_1,\g_2) = (\G_{1\vu},\G_{2\vu})$.
    \end{itemize}
    \item {\bf Second case:} $\G_{1\vu} > \g_1^{\max}$ and $\G_{2\vu} > \g_2^{\max}$. Then, $(\g_1,\g_2) = (\g_1^{\max}, \g_2^{\max})$.
    \item {\bf Third case:} $\G_{1\vu} > \g_1^{\max}$ and $\G_{2\vu} \leq \g_2^{\max}$;
    \begin{itemize}
        \item if $\au<\bu$, then 
$$
 \g_2=\operatorname{proj}_{[0,\g_2^{\max}]}\left(\frac{\g_{3\vu}^{\max}-\au\g_1^{\max}}{\bu}\right), \quad  
            \g_1=\left\{\begin{array}{ll}
            \g_3^{\max}/\au & \text{if }\g_2=0\\
            \g_1^{\max} & \text{if }\g_2\in(0,\g_2^{\max}];
            \end{array}\right. 
$$
        \item else,
$$
\g_2=\operatorname{proj}_{[0,\g_2^{\max}]}\left(\frac{\g_{4\vu}^{\max}-(1-\au)\g_1^{\max}}{1-\bu}\right), \quad 
\g_1=\left\{\begin{array}{ll}
\g_4^{\max}/(1-\au) & \text{if }\g_2=0 \\
 \g_1^{\max}& \text{if }\g_2\in(0,\g_2^{\max}];
            \end{array}\right. 
  $$
    \end{itemize} 
    \item {\bf Fourth case:} $\G_{1\vu}\leq\g_1^{\max}$ and $\G_{2\vu}>\g_2^{\max}$;
    \begin{itemize}
        \item if $\au>\bu$, then 
$$
 \g_1=\operatorname{proj}_{[0,\g_1^{\max}]}\left(\frac{\g_{3\vu}^{\max}-\bu\g_1^{\max}}{\au}\right), \quad 
\g_2=\left\{\begin{array}{ll}
\g_3^{\max}/\bu \text{if } \g_1=0 \\
 \g_2^{\max} & \text{if }\g_1\in(0,\g_1^{\max}]; 
            \end{array}\right. 
  $$
         \item else,
$$
                \g_1=\operatorname{proj}_{[0,\g_1^{\max}]}\left(\frac{\g_{4\vu}^{\max}-(1-\bu)\g_1^{\max}}{1-\au}\right), \quad
                \g_2=\left\{\begin{array}{ll}
                \g_4^{\max}/(1-\bu) & \text{if }\g_1=0 \\
                \g_2^{\max}& \text{if } \g_1\in(0,\g_1^{\max}]. 
   \end{array}\right. 
   $$
    \end{itemize}
\end{itemize}

\paragraph{Regular approximation of $\phi^{LP}$.}
In the following, in order to use differentiable optimization techniques, we will replace the function $\phi^{LP}$, which is only Lipschitz, by an approximation $C^1$, replacing all operations consisting in taking the minimum or the maximum of two quantities by regular approximations:
$$
\min\{x,y\}\leftarrow \frac{x+y-\sqrt{(x+y)^2+\eta^2}}{2}\quad \text{and}\quad \max\{x,y\}\leftarrow \frac{x+y+\sqrt{(x+y)^2+\eta^2}}{2}
$$ 
for a given $\eta>0$.
\begin{remark}[Some remarks on the LP problem]
It is notable that In \cite{Canic15}, a close LP problem has been solved. However, we found that a case corresponding to a particular constraint configuration was forgotten in their analysis.
This error could lead to erroneous vertices providing negative values to some flows. The configuration that was not taken into account in this publication is illustrated on~Fig.~\ref{fig:lp2x2corrected}, for parameter choices from~\eqref{tab:lp2x2corrected}.

\begin{table}[h]
    \centering
    \begin{tabular}{c|c|c|c|c|c|c|c|c}
        \toprule
        Parameter & $\g_1^{\max}$ & $\g_2^{\max}$ & $\g_3^{\max}$ & $\g_4^{\max}$ & $u_3$ & $u_4$ & $\a$ & $\b$ \\
        \midrule
        Value     & 0.25 & 0.25 & 0.2244 & 0.2244 & 0 & 0.96 & 0.45 & 0.50  \\
        \bottomrule
    \end{tabular}
    \caption{Parameters for the 2x2 LP problem in~\eqref{fig:lp2x2corrected}.}
    \label{tab:lp2x2corrected}
\end{table}

\begin{figure}[h]
\centering
\begin{subfigure}[t!]{0.4\textwidth}
    \begin{tikzpicture}[line cap=round,line join=round,>=triangle 45,x=7.0cm,y=7.0cm]
    \draw[->,color=black] (-0.1,0) -- (1.3,0);
    \draw[->,color=black] (0,-0.9) -- (0,0.32);
    \draw[color=black] (0pt,-10pt) node[right] {\footnotesize $0$};
    \clip(-0.1,-0.9) rectangle (1.3,0.32);
    \fill[color=zzttqq,fill=zzttqq,fill opacity=0.1] (0,0) -- (0.25,0) -- (0.25,0.25) -- (0,0.25) -- cycle;
    \draw [domain=-0.1:1.3] plot(\x,{(--0.25-0*\x)/1});
    \draw (0.25,-0.9) -- (0.25,0.32);
    \draw [color=qqwwff,domain=-0.1:1.3] plot(\x,{(--0.45-1.2*\x)/1});
    \draw [color=ffwwqq,domain=-0.1:1.3] plot(\x,{(--0.23-1*\x)/1});
    \draw [color=zzttqq] (0,0)-- (0.25,0);
    \draw [color=zzttqq] (0.25,0)-- (0.25,0.25);
    \draw [color=zzttqq] (0.25,0.25)-- (0,0.25);
    \begin{scriptsize}
    \fill [color=zzqqzz] (1.1,-0.87) circle (1.5pt);
    \draw[color=zzqqzz] (1.12,-0.85) node {$\vG_\vu$};
    \fill [color=qqzztt] (0.23,0) circle (2.0pt);
    \draw[color=qqzztt] (0.23,0.03) node {$\bm{S}$};
    \fill [color=ffqqqq] (0.25,-0.02) circle (2.0pt);
    \draw[color=ffqqqq] (0.267,-0.05) node {$\bm{F}$};
    \draw[color=ffwwqq] (0.60,-0.45) node {$\mathcal{C}_3$};
    \draw[color=qqwwff] (0.70,-0.30) node {$\mathcal{C}_4$};
    \end{scriptsize}
    \end{tikzpicture}
\end{subfigure}
\hfill 
\begin{subfigure}[t!]{0.35\textwidth}
    \begin{tikzpicture}[line cap=round,line join=round,>=triangle 45,x=50.0cm,y=50.0cm]
    \draw[->,color=black] (0.2,0) -- (0.3,0);
    \clip(0.2,-0.05) rectangle (0.3,0.05);
    \fill[color=zzttqq,fill=zzttqq,fill opacity=0.1] (0,0) -- (0.25,0) -- (0.25,0.25) -- (0,0.25) -- cycle;
    \draw [domain=0.2:0.3] plot(\x,{(--0.25-0*\x)/1});
    \draw (0.25,-0.05) -- (0.25,0.05);
    \draw [color=qqwwff,domain=0.2:0.3] plot(\x,{(--0.45-1.2*\x)/1});
    \draw [color=ffwwqq,domain=0.2:0.3] plot(\x,{(--0.23-1*\x)/1});
    \draw [color=ffqqqq](1.64,0.17) node[anchor=north west] {$levelset \, = \,0.23$};
    \draw [color=zzttqq] (0,0)-- (0.25,0);
    \draw [color=zzttqq] (0.25,0)-- (0.25,0.25);
    \draw [color=zzttqq] (0.25,0.25)-- (0,0.25);
    \begin{scriptsize}
    \fill [color=zzqqzz] (1.1,-0.87) circle (1.5pt);
    \fill [color=qqzztt] (0.23,0) circle (2.0pt);
    \draw[color=qqzztt] (0.235,0.005) node {$\bm{S}$};
    \fill [color=ffqqqq] (0.25,-0.02) circle (2.0pt);
    \draw[color=ffqqqq] (0.255,-0.015) node {$\bm{F}$};
    \end{scriptsize}
    \end{tikzpicture}
\end{subfigure}
\caption{Graphical resolution of the $2\times 2$ LP problem. polytope of constraint (left) and a zoom (right). $\bm{F}$ is the erroneous solution and $\bm{S}$ is the correct one.}
\label{fig:lp2x2corrected}
\end{figure}
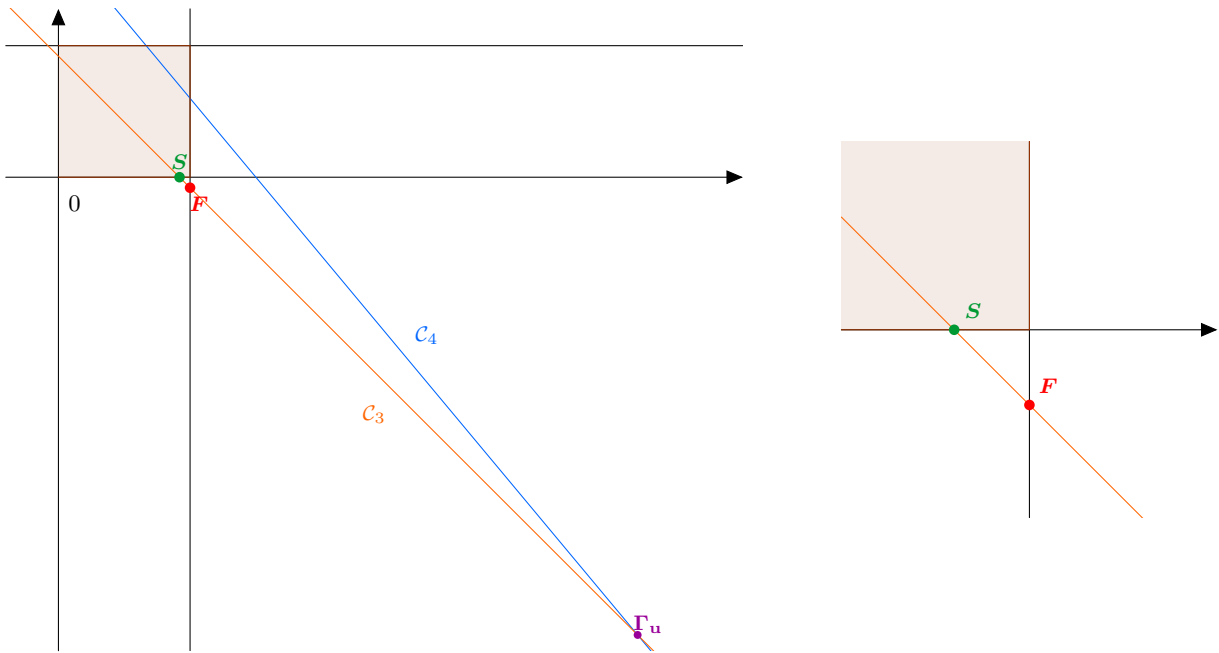
\end{remark}

\bibliographystyle{abbrv} 
{\footnotesize
\bibliography{trafic_BFNP}}

\end{document}